\documentclass[a4paper]{amsart}
\usepackage[utf8]{inputenc}
\usepackage[english]{babel}
\usepackage{fancyhdr}
\usepackage{amsmath}
\usepackage{amsfonts,amstext}
\usepackage{amssymb,amsthm,amscd,amsxtra,wasysym,graphicx}
\usepackage{mathrsfs,esint,comment}
\usepackage{tikz-cd}
\usepackage{enumitem,mathtools,dsfont}
\usepackage{palatino,mathpazo}
\usepackage{charter}
\usepackage{xcolor}

\def\Bibtex{{\rm B\kern-.05em{\sc i\kern-.025em b}\kern-0.08em T\kern-.1667em\lower.7ex\hbox{E}\kern-.125emX}}
 \usepackage{hyperref}
\hypersetup{
     unicode=false,          
    pdftoolbar=true,       
    pdfmenubar=true,       
    pdffitwindow=false,     
    pdfstartview={FitH}, 
    pdftitle={Monge-Amp\`ere equations},   
    pdfauthor={Quang-Tuan Dang},   
    colorlinks=true,  
   linkcolor=purple,         
    citecolor=blue,        
    filecolor=green,      
    urlcolor=blue}          

\usepackage{cleveref}
\usepackage{indentfirst}
\theoremstyle{plain}
\newtheorem{theorem}{Theorem}
\newtheorem{lemma}[theorem]{Lemma}

\newtheorem{proposition}[theorem]{Proposition}

\theoremstyle{definition}
\newtheorem{definition}[theorem]{Definition}
\newtheorem{remark}[theorem]{Remark}

\theoremstyle{plain}
\newtheorem{bigthm}{Theorem}

\numberwithin{theorem}{section}
\numberwithin{equation}{section}
 
\DeclareMathOperator \PSH {{\rm PSH}}

\DeclareMathOperator \Vol {{\rm Vol}}

\DeclareMathOperator \exph {{\rm exph}}

\DeclareMathOperator \tr {\textrm{tr}}
\newcommand\om{\omega}

\def\om{\omega}

\def\f{\varphi}
\def\dc{dd^c}

\begin{document}
		\title[Regularity of solutions to Monge--Amp\`ere equations]{Regularity of solutions to Monge--Amp\`ere equations on compact Hermitian manifolds}
	\author{Quang-Tuan Dang}

		\address{Yau Mathmatical Sciences Center, Tsinghua University, Beijing, China 100084}
		\email{dangquangtuan10@gmail.com $\&$ dangqt@mail.tsinghua.edu.cn}
	\date{\today}
	\keywords{Complex Monge--Amp\`ere equations, Hermitian manifolds}
	\subjclass[2020]{32U20, 32W20, 32U05}
	
\begin{abstract}
    We study the stability and H\"older continuity of solutions to degenerate complex Monge--Ampère equations associated with a non-closed big \((1,1)\)-form on compact Hermitian manifolds. We also show that the solution is globally continuous when the reference form is the pullback of a Hermitian metric. As a consequence, we establish a uniform diameter bound for the twisted Chern--Ricci flow.
\end{abstract}
\maketitle
\tableofcontents
\section{Introduction}\label{sect: intro}
Complex Monge--Amp\`ere equations have played a central role in complex geometry since the celebrated work of Yau~\cite{yau1978ricci} solving the Calabi conjecture. The study of such equations on compact Hermitian manifolds, with a fixed Hermitian background metric, was initiated several decades ago and has undergone significant development in recent years; see, for example,~\cite{cherrier1987equation,hanani1996equation,hanani1996generalisation,tosatti2010estimates,tosatti2010complex,dinew2012pluripotential,kolodziej2015weak,nguyen2016complex} and the references therein. More recently, the degenerate metric setting has been investigated by Guedj--Lu~\cite{guedj2022quasi,guedj2021quasi} and Boucksom--Guedj--Lu \cite{BoucksomGuedjLu2025-volume}, providing a very general framework that encompasses many geometric applications.
The regularity of solutions to degenerate complex Monge--Amp\`ere equations has important consequences in complex dynamics and geometry; see, for instance,~\cite{dinh2014characterization,Vu2018complex,Fu-Guo-Song2020-geometric,Li_yang_2021,GuoPhongTongWang2021-modulus,guo2022-local,guo2022-diameter,GuedjGuenanciaZeriahi25-diameter,Guo2024-diameter2,do2023log,vu24-diameter,nguyen-vu24-diameter} and references therein.

\medskip
Let $X$ be a compact complex $n$-dimensional manifold equipped with a Hermitian metric $\omega_X$. Let $\theta$ be a smooth real (1,1) form on $X$.
We let $\PSH(X,\theta)$ be the set of $\theta$ -plurisubharmonic functions ($\theta$-psh for short) $\varphi$ such that $\varphi$ is locally the sum of a plurisubharmonic function and a smooth function, and $\theta+\dc \f\geq 0$ in the weak sense of currents. Here, we put $d^c=\frac{i}{2}(\Bar{\partial}-\partial)$ so that $\dc=i\partial\Bar{\partial}$.

We say that $\theta$ is {\em big} if there exists a $\theta$-psh function $\psi_0$ with analytic singularities (see Definition~\ref{def: analytic}) such that $\theta+\dc\psi_0\geq \varepsilon_0\omega_X$ for some $\varepsilon_0>0$. We let $\Omega$ denote the open Zariski set where $\psi_0$ is locally bounded and $\psi_0=-\infty$ on $\partial\Omega$.

 Following Bedford-Taylor's theory~\cite{bedford1976dirichlet,bedford1982new}, it was shown (see, e.g.,~\cite{dinew2012pluripotential,kolodziej2015weak}) that for any $\f\in\PSH(X,\theta)\cap L_{\rm loc}^\infty(\Omega)$
 the complex Monge--Amp\`ere operator $$(\theta+\dc\f)^n$$ is a well-defined positive Borel measure in $\Omega$. 
It is therefore meaningful to study the complex Monge--Amp\`ere equation
\begin{equation}\label{eq: cmae}
    (\theta+\dc\f)^n=c\mu\quad\text{in}\;\Omega,
\end{equation} for a given positive measure $\mu$ and $c$ a normalization constant.

When $\mu = fdV_X$ is absolutely continuous with respect to the Lebesgue measure $dV_X$, with density $f \in L^p(X)$ for some $p>1$, generalizing the result of~\cite{boucksom2010monge}, it was proved known in~\cite[Theorem D]{BoucksomGuedjLu2025-volume} that there exists a solution $(\varphi,c)\in \PSH(X,\theta)\times \mathbb{R}_{>0}$ to~\eqref{eq: cmae}. Moreover, such a solution is continuous on the Zariski open set $\Omega$ (see~\cite[Theorem 3.7]{guedj2021quasi}). In this setting, H\"older continuity is typically the strongest regularity one can expect.

\begin{bigthm} \label{thmA}  Let $(X,\omega_X)$ be a compact Hermitian manifold of complex dimension $n$. Let $\theta$ be a big smooth real (1,1)-form.
    Let $\mu=fdV_X$ be a measure absolutely continuous with respect to Lebesgue measure with density $0\leq f\in L^p(X,dV_X)$,  $p>1$. 
Let $\varphi\in \PSH(X,\theta)$ and $c\in  (0,+\infty)$ be such that $\sup_X\varphi=0$,  $$V_\theta-C_0\leq\varphi\leq V_\theta,\qquad(\theta+\dc\f)^n=c\mu\; \text{in}\,\Omega,$$ with a uniform constant $C_0>0$. Then $\f$ is H\"older continuous in $\Omega$.
\end{bigthm}
Here, $V_\theta=\sup\{u\in\PSH(X,\theta): u\leq 0 \}$ is a $\theta$-psh function with minimal singularities.
We note that the constant $c>0$ is uniformly bounded in terms of $n$, $p$, $\omega_X$, and a lower bound for $\int_X f^{1/n} dV_X$; see~\cite[Theorem 4.7]{BoucksomGuedjLu2025-volume}. The proof of Theorem~\ref{thmA} relies on the stability result (cf. Theorem~\ref{thm :stability-big}) together with the arguments of~\cite[Theorem D]{demailly2014holder}, which are based on the Demailly regularization technique. Our proof of the stability result is based on the use of auxiliary equations, an approach that can also be applied in the local setting.
We refer interested readers to~\cite{guedj2008holder,eyssidieux2009singular,guedj2012stability,kolodziej2018holder,kolodziej2021continuous,GGZ23-continuity,guedj2021quasi1} for the stability estimate obtained via the pluripotential approach, and to~\cite{WWZ20-estimate,GuoPhongTongWang2021-modulus,WZ24-trace,Cheng-Xu24-m-subharmonic,ChengXu2025-viscosity-hessian} and the references therein for the stability estimate established by the PDE methods.

The higher regularity of solutions on the Zariski open set $\Omega$, when $f$ is smooth, is an important open problem, which remains open even in the K\"ahler case; cf. \cite{boucksom2010monge}. Under the additional assumption that $\theta$ is semipositive, Guedj--Lu~\cite[Theorem 4.1]{guedj2021quasi} showed that the solution is smooth on $\Omega$. However, the question of global regularity is still widely open.

\medskip
Next, we extend the results of Dinew--Zhang~\cite{Dinew-Zhang10-stability} and Cho--Choi~\cite{ChoChoi25-continuity} to the Hermitian setting in order to study the global continuity of solutions when $\theta$ is the pullback of a Hermitian metric. More precisely, we prove the following.


\begin{bigthm}\label{thmB}
    Let $V$ be a compact complex variety of dimension $n$
with log terminal singularities, equipped with  a Hermitian form $\omega_V$. Let $\pi:X\to V$ be a log resolution of singularities.   Set $\theta=\pi^*\om_V$. Let $0\leq f\in L^p(X,dV_X)$ for some $p>1$.
Assume that $\varphi\in \PSH(X,\theta)\cap L^\infty(X)$ and $c\in\mathbb R_+$ solve the following complex Monge--Amp\`ere equation  $$\sup_X\varphi=0,\qquad(\theta+\dc\f)^n=cfdV_X.$$
Then $\varphi$ is continuous on $X$.
\end{bigthm}
The existence of a globally bounded solution was established in~\cite{guedj2021quasi}, where the solution is also shown to be smooth on $\pi^{-1}(V^{\mathrm{reg}})$. The original idea of the proof of Theorem~\ref{thmB} goes back to~\cite{kolodziej1998complex}.

\medskip 
Finally, we apply our continuity results to study the diameter bounds along the Chern–-Ricci flow, similar to~\cite{deruelle2025k} for the K\"ahler–Ricci flow. Generalizing the work of Guedj--Zeriahi~\cite{guedj2017regularizing} and Di Nezza--Lu~\cite{di2017uniqueness}, T\^o \cite{to2018regularizing} showed that there exists a unique twisted Chern-–Ricci flow on $X$ with initial condition $T_0$, that is, a smooth family of Hermitian metrics $(\omega_t)_{t\in[0,T_{\max})}$ that satisfies 
 \begin{equation}\label{eq: crf0}
     \frac{\partial\omega_t}{\partial t}=-\textrm{Ric}(\omega_t)+\eta ,\quad \omega_t\xrightarrow{t\to 0}T_0\;\text{weakly},
 \end{equation}
where $\mathrm{Ric}(\omega_t)$ denotes the Chern-–Ricci form of $\omega_t$, $\eta$ is a smooth $(1,1)$-form, and $T_0=\omega_X+\dc \varphi_0$ is a positive $(1,1)$-current with $\varphi_0\in L^\infty(X)$. The maximal existence time is given by
 \[ T_{\max}=\sup\{t>0:\exists\,\hat{\psi}_t\in\mathcal{C}^\infty(X),\omega_X+t(\eta-\textrm{Ric}(\omega_X))+\dc\hat\psi_t>0 \}.\]
 Note that the convergence at $t=0$ holds in the weak sense of currents. We further investigate diameter bounds for solutions to the twisted Chern--Ricci flow~\eqref{eq: crf0} under geometric assumptions on the initial data $T_0$.
   
\begin{bigthm}\label{thmC} Let $(X,\omega_X)$ be a compact Hermitian manifold of complex dimension $n$.
Assume that $T_0=\omega_X+\dc\varphi_0$ with $\varphi_0\in\mathcal{C}^0(X)$ and $T_0^n=e^{\psi^+-\psi^-}\omega_X^n$ where $\psi^{\pm}\in\PSH(X,A\omega_X)$ for $A>0$ fixed, and $e^{-\psi^-}\in L^p(\omega_X^n)$, for $p>1$.
Let $(\omega_t)_{t\in [0,T_{\max})}$ be a solution to the weak twisted Chern--Ricci flow~\eqref{eq: crf0} starting at $T_0$. 
    Then for any $x,y\in X$,
        \[ \textrm{diam}(X,\omega_t)\leq C,\qquad d_{\omega_t}(x,y)\leq Cd_{\omega_X}(x,y)^\alpha,\;\forall\, t\in (0,T_{\max}/2],\]
        for some constants $C,\alpha>0$ that only depend on $X$, $A$, $\omega_X$, $p$ and an upper bound for $\|e^{-\psi^-}\|_p$, where $d_{\omega_t}$ denotes the  distance with respect to $\omega_t$.
\end{bigthm}
The diameter bound of the families of K\"ahler metrics has been  studied in~\cite{Fu-Guo-Song2020-geometric,Li_yang_2021,GuoPhongTongWang2021-modulus,guo2022-local,guo2022-diameter,GuedjGuenanciaZeriahi25-diameter,Guo2024-diameter2,Guo-Phong-Sturm24-green,vu24-diameter,nguyen-vu24-diameter,GJSS25-cscK}. We follow the same path as in~\cite{Li_yang_2021} using H\"older regularity
of the Monge--Ampère potentials to establish a uniform upper bound on diameters; cf.~Proposition ~\ref{prop: diameter}.

\medskip 
The paper is organized as follows. In Section~\ref{sect: pre} we recall some necessary materials which come from pluripotential theory in the Hermitian setting. Theorem~\ref{thmA} on the local H\"older continuity of solutions is proved
in Section~\ref{sect: local}, where we also establish the stability property.
Theorem~\ref{thmB} on the global continuity of solutions is proved in
Section~\ref{sect: global}. Finally, Section~\ref{sect: application} is devoted to the diameter bounds along
the weak Chern--Ricci flow and contains the proof of Theorem~\ref{thmC}.


 \subsection*{Acknowledgment} The author thanks
B. Zhou for discussions and the anonymous referee for his comments which helped improve
the paper. The author is supported by  the Shuimu Scholar program of Tsinghua University.
   \subsection*{AI disclosure} 
  The author used ChatGPT to improve the presentation of the manuscript and to assist in identifying inconsistencies in notation and minor errors. The author assumes full responsibility for the mathematical content of the paper.
\section{Preliminaries}\label{sect: pre} In the paper, $X$ denotes a compact Hermitian manifold of complex dimension $n$, equipped with a Hermitian form $\omega_X$. 
	Let \(dV_{X}:={\omega_X^n}/n! \)
	denote the volume form associated with $\omega_X$. For any $p\geq 1$, we simply write $\|f\|_p$ for the $L^p(X,dV_X)$-norm of $f$.

\subsection{Quasi-psh functions and Monge--Amp\`ere measures}

Recall that an upper semi-continuous function $ \varphi:X \rightarrow\mathbb{R}\cup\{-\infty\} $
	is called {\it quasi-plurisubharmonic} ({\it quasi-psh} for short) if it is locally the sum of a smooth and a plurisubharmonic (psh for short) function. In particular, $\varphi$ is usc and integrable. Quasi-psh functions are actually in $L^p(X,dV_X)$
for any $p\in[1,\infty)$, and the induced topologies are also equivalent; see, e.g.,~\cite{demaillycomplex,guedj2017degenerate}.
\begin{definition}\label{def: analytic}
  We say that a quasi-psh function $\varphi$ has {\em analytic singularities} if $\varphi$ can be locally written as
\[\varphi=c\log \sum_{j=1}^N|f_j|^2 +R, \]
where $c>0$, the $f_j$'s are holomorphic functions, and $R$ is a bounded function.   
\end{definition}

Let $\theta$ be a real smooth (1,1) form on $X$. We say that $\varphi$ is {\it $\theta$-plurisubharmonic}  ({\it $\theta$-psh} for short) if it is quasi-psh, and $$\theta_\varphi:=\theta+\dc \varphi\geq 0$$ in the sense of currents, where ${\rm d}=
	\partial+\Bar{\partial}$ and ${\rm d}^c=\frac{i}{2}(\Bar{\partial}-\partial)$ so that $\dc={i}\partial\Bar{\partial}$. Let $\PSH(X,\theta)$ denote the set of all $\theta$-psh functions which are not identically $-\infty$.

    We say that $\theta$ is {\em big} if there exists a $\psi_0\in\PSH(X,\theta)$ such that $\theta+\dc \psi_0\geq \varepsilon_0\omega_X$ for some $\varepsilon_0>0$. As a consequence of Demailly--P\u aun’s regularization~\cite{demailly2004numerical}, we can choose $\psi_0$ to have analytic singularities. 

    We denote by $$V_\theta=\sup\{u\in\PSH(X,\theta):u\leq 0 \}$$ a $\theta$-psh function with minimal singularities. We observe that $V_\theta\geq \psi_0$, in particular $V_\theta$ is locally bounded in $\Omega$. 

    Let $U$ be an open subset of $X$. Let $\varphi\in\PSH(X,\theta)\cap L_{\rm loc}^\infty(U)$.
    Following the construction of Bedford-Taylor~\cite{bedford1976dirichlet,bedford1982new}, it has been shown in~\cite{dinew2012pluripotential,kolodziej2015weak} that the complex Monge--Amp\`ere measure \[\theta^n_\varphi=(\theta+\dc \varphi)^n\] is well-defined on $U$. 

We recall the following domination principle, which will be useful in the sequel.
    \begin{proposition}[{\cite[Lemma 4.2]{BoucksomGuedjLu2025-volume}}]\label{prop: domination} Let $\theta$ be a big (1,1) form and $\psi_0\in\PSH(X,\theta)$ with analytic singularities such that $\theta+\dc \psi_0\geq \varepsilon_0\omega_X$ for some $\varepsilon_0>0$.
        Let $\varphi,\psi$ be $\theta$-psh functions on $X$ such that $\min\{\varphi,\psi\}\geq \psi_0$ and $\psi\leq \varphi+C$ for some $C>0$. If \[(\theta+\dc \varphi)^n\leq c(\theta+\dc \psi)^n \; \text{on}\;\{\varphi<\psi\}\cap \{\psi_0>-\infty\}, \] for some $c\in[0,1)$, then $\varphi\geq \psi$. 
    \end{proposition}
\subsection{Demailly's regularization}
Following~\cite{demailly1994regularization}, we consider $\rho_\delta\f$-regularization of the function $\f$ defined by
\[\rho_\delta\f(z)=\frac{1}{\delta^{2n}}\int_{\zeta\in T_zX}\f(\exph_z(\zeta))\rho\left( \frac{|\zeta|^2_{\omega_X}}{\delta^2}\right)dV_{\omega_X}(\zeta), \quad\delta>0,\] where $\zeta\mapsto\exph_z(\zeta)$
 is the (formal) holomorphic part of the Taylor expansion
of the Riemann exponential map of the Chern connection on the tangent bundle of $X$
associated to $\omega$ and the smoothing kernel $\rho:\mathbb{R}_+\to\mathbb{R}_+$ is given by
\[\rho(t)=\begin{cases}
    \frac{\eta}{(1-t)^2}\exp\left(\frac{1}{t-1} \right), & 0\leq t\leq 1    \\
    0, & t>1,
\end{cases}\] with a normalizing constant $\eta$ such that $\int_\mathbb{R}\rho(t)dt=1$.  We define the Kiselman--Legendre transform:
\begin{equation}\label{eq: Kiselman}
   \Phi_{c,\delta}:=\inf_{t\in[0,\delta]} \left(\rho_t\f +Kt-K\delta-c\log\frac{t}{\delta}\right),  
\end{equation}
where $c>0$, $\delta\in (0,1)$ and $K$ is a positive (curvature) constant (as in~\cite{demailly1994regularization}) such that $t\to \rho_t\f+Kt$
increases in $t$.
Following~\cite[Lemma 4.1]{kolodziej2019stability} we obtain the following
\begin{equation}\label{eq: regularization}
    \theta+\dc\Phi_{c,\delta}\geq -(Ac+2K\delta)\omega_X,
\end{equation} where $A$ is a lower bound of the negative part of the Chern curvature of $\omega_X$. 


By~\cite[Lemma 2.3]{demailly2014holder}, there exists a constant $C>0$ depending on $\omega_X$ such that \begin{equation}\label{eq: integral}
    \int_X\frac{|\rho_\delta\varphi-\varphi|}{\delta^2}dV_X\leq C\|\varphi\|_{L^1(dV_X)}.
\end{equation}
We remark that although the lemma is stated for K\"ahler manifolds,
the same proof works for Hermitian ones after replacing the Riemann curvature tensor
by the Chern curvature tensor used in~\cite{demailly1994regularization}.

\section{Local H\"older continuity}\label{sect: local}

In this section, we assume $\theta$ is a real smooth {\em big} (1,1) form, i.e., there exists a $\theta$-psh function $\psi_0$ with {\em analytic singularities} satisfying $$\theta+\dc\psi_0\geq \varepsilon_0\omega_X$$ in the sense of currents for some $\varepsilon_0>0$.  
We observe that $\psi_0$ is locally bounded in an open Zariski subset $\Omega:=X\setminus\{\psi_0=-\infty\}$. By subtracting a positive constant, we may assume that $\psi_0\leq V_\theta$.

Let $0\leq f\in L^p(X,dV_X)$ where $p>1$. It follows from ~\cite[Theorem D]{BoucksomGuedjLu2025-volume} that there exist a unique constant $c=c(\theta,f)>0$ and a  $\theta$-psh function such that $\sup_X\varphi=0$,
\begin{equation}\label{eq: cmae-big}
V_\theta-C\leq\varphi\leq V_\theta\quad\text{and}\quad    (\theta+\dc\varphi)^n=cfdV_X\quad \text{in}\; \Omega,
\end{equation} for $C>0$ depending on $\omega_X$, $\theta$, $p$, and $\|f\|_p$.

\smallskip
In this context, we prove the following.
\begin{theorem}\label{thm: local-Holder}
Let $\mu=fdV_X$ be a measure absolutely continuous with respect to Lebesgue measure with density $0\leq f\in L^p(X,dV_X)$,  $p>1$. 
Let $(\varphi,c)\in \PSH(X,\theta)\times (0,+\infty)$ be such that $\sup_X\varphi=0$,  $$V_\theta-C_0\leq\varphi\leq V_\theta,\qquad(\theta+\dc\f)^n=c\mu\; \text{in}\,\Omega,$$ for $C_0>0$ depending only on $\omega_X$, $X$, $\theta$, $p$, and $\|f\|_p$. Then  for any $0<\alpha<\frac{2p-2}{np+p-1}$, $\f$ is H\"older continuous with exponent $\alpha$ in $\Omega$.
\end{theorem}

\subsection{Stability}
We recall the following lemma, which is used in the proof of the stability result.
\begin{lemma}[{\cite{BoucksomGuedjLu2025-volume}}]\label{lem: subsolution}
    Let $g$ be a measurable function on $X$ such that $g\in L^q(X,dV_X)$ for some $q>1$. Then there exists
$v\in\PSH(X,\theta)$ with $V_\theta-1\leq v\leq V_\theta$ such that
\[(\theta+\dc v)^n \geq cgdV_X\;\;\text{on}\,\Omega,\] where $c>0$ only depends on $X$, $\omega_X$, $q$, $\theta$, and a lower bound for $\|g\|_q^{-1}$. 
\end{lemma}
\begin{proof}
   Applying \cite[Lemma 4.5]{BoucksomGuedjLu2025-volume} with $V_\theta$ and $\Tilde{g}=\frac{g}{\|g\|_q}$, we obtain the desired inequality.
\end{proof}

We establish the following stability result, which is analogous to~\cite[Theorem C]{guedj2012stability} for the K\"ahler case.

\begin{theorem}\label{thm :stability-big} Let $\mu=fdV_X$ be  a positive measure absolutely continuous with respect to the Lebesgue measure with density $f\in L^p(X,dV_X)$,  $p>1$.
Let $\f$ be a $\theta$-psh function such that $$\sup_X\f=0,\quad (\theta+\dc\f)^n=c\mu\;\text{on}\,\Omega,$$ for some $c>0$ and $V_\theta-C_0\leq \varphi\leq V_\theta$, with $C_0>0$. 
Assume $0\geq\psi\in\PSH(X,\theta)$. Then there exists $\gamma_0$ depending on $n$, $p$ such that for any $0<\gamma<\gamma_0:=\frac{p-1}{np+p-1}$,
\[\sup_X(\psi-\f)_+\leq C\|(\psi-\f)_+\|_{1}^\gamma,\]
where $C$ depends only on $X$, $\omega_X$, $\theta$, $c$, $C_0$, $p$ and an upper bound for $\|f\|_{p}$.
\end{theorem}

The proof is based on the use of auxiliary subsolutions and the comparison principle, an approach that arose from discussions with Chinh H. Lu at AIMS, Senegal. We also provide an alternative proof in the K\"ahler case that avoids the use of the relative Monge--Amp\`ere capacity. 
\begin{proof}
We observe that $(\psi-\varphi)_+=\max(\psi,\varphi)-\varphi$. We may assume that $\psi\geq \varphi$.

Fixing any $q>1$, we can find $q>r>1$ such that $\frac{1}{r}=\frac{1}{p}\big(1-\frac{1}{q}\big)+\frac{1}{q}$. 
We will prove the following stability estimate
\begin{equation}\label{est: stability}
    \sup_X(\psi-\varphi)\leq \delta+ C[\mu(\{\varphi<\psi-\delta\})]^{\frac{1}{nq}},
\end{equation} for $C>0$ only depending on $X$, $\omega_X$, $q$, $\theta$, and a lower bound for $\|f\|_q^{-1}$. 
Since the case $\mu(\varphi<\psi-\delta)=0$ is trivial, we assume that $\mu(\{\varphi<\psi-\delta\})>0$.
     If we set $F=\frac{\mathbf{1}_{\{\varphi<\psi-\delta\}}f}{[\mu(\{\varphi<\psi-\delta\})]^{\frac{1}{q}}}$ then
     the H\"older inequality yields
     \begin{align*}
    \int_X F^rdV_X&=\int_X f^{r-\frac{r}{q}} \bigg(\frac{\mathbf{1}_{\{\varphi<\psi-\delta\}}f}{\mu(\{\varphi<\psi-\delta\})}\bigg)^{\frac{r}{q}} dV_X\\
   & \leq \|f\|^{r-\frac{r}{q}}_p\bigg( \int_X \frac{\mathbf{1}_{\{\varphi<\psi-\delta\}}f}{\mu(\{\varphi<\psi-\delta\})}dV_X \bigg)^{\frac{r}{q}}=\|f\|_p^{r-\frac{r}{q}}.
\end{align*} We apply Lemma~\ref{lem: subsolution} to obtain a constant $c>0$ and $v\in\PSH(X,\theta)$ such that $V_\theta-1\leq v\leq V_\theta$ and 
$$(\theta+\dc v)^n\geq c^n\frac{ \mathbf{1}_{\{\varphi<\psi-\delta\}}f dV_X}{[\mu(\{\varphi<\psi-\delta\})]^{\frac{1}{q}}},$$
for $c>0$ only depending on $\omega_X$, $p$, $q$, $\theta$ and a lower bound for $\|f\|_p^{-1}$. 


Set $u:=(1-\lambda)\psi+\lambda v$ for $\lambda\in (0,1)$ to be determined later. Since $V_\theta\leq \psi+C_0$, we infer that
\[\{\varphi<u-\delta -C_0\lambda\} \subset \{\varphi<\psi-\delta \}.\]
We have
\[\mathbf{1}_{\{\varphi<u-\delta -C_0\lambda\}}\theta_u^n\geq \lambda^n\mathbf{1}_{\{\varphi<u-\delta -C_0\lambda\}}\theta_v^n\geq \lambda^n c^n\frac{\mathbf{1}_{\{\varphi<u-\delta -C_0\lambda\}}f dV_X}{[\mu(\{\varphi<\psi-\delta\})]^{\frac{1}{q}}}.\]
If $\mu(\{\varphi<\psi-\delta\})^{\frac{1}{q}}\geq c^n/2$, then the stability property~\eqref{est: stability} follows trivially.
Otherwise, $[\mu(\{\varphi<\psi-\delta\})]^{\frac{1}{q}}<  c^n/2$, hence we can choose $$\lambda=2^{1/n}c^{-1}[\mu(\{\varphi<\psi-\delta\})]^{\frac{1}{qn}}\in (0,1)$$ to obtain that
\[\mathbf{1}_{\{\varphi<u-\delta -C_0\lambda\}}\theta^n_u\geq 2\mathbf{1}_{\{\varphi<u-\delta -C_0\lambda\}}fdV_X=\mathbf{1}_{\{\varphi<u-\delta -C_0\lambda\}}2\theta_\varphi^n. \]
We observe that $\varphi+C_0\geq V_\theta\geq u-\delta-C_0\lambda$.
 Thus, it follows from the domination principle (Proposition~\ref{prop: domination}) that $\varphi\geq u-\delta-C_0\lambda$ on $X$, hence
\[\psi-\varphi\leq\delta+C_\lambda+\lambda(\psi-v) \leq 
\delta+(C_0+1)\lambda. \]
This implies the desired estimate~\eqref{est: stability} because of the choice of $\lambda$. By the Chebyshev inequality and the H\"older inequality, we have
\begin{equation}\label{eq: Chebyshev}
    \mu(\{\varphi<\psi-\delta\})\leq\frac{1}{\delta^{\frac{p-1}{p}}} \int_X (\psi-\varphi)_+^{\frac{p-1}{p}}fdV_X\leq \frac{\|f\|_p\cdot\|(\psi-\varphi)\|_1^{\frac{p-1}{p}}}{\delta^{\frac{p-1}{p}}}.
\end{equation} 
Choosing $\delta:=\|(\psi-\varphi)_+\|^\gamma$ for $\gamma=\frac{1}{q}\frac{p-1}{np+p-1}$, it follows from \eqref{est: stability} and \eqref{eq: Chebyshev} that
\[\sup_X(\psi-\varphi)\leq (C+1)\|(\psi-\varphi)_+\|_1^\gamma. \]
Since $q>1$ was arbitrarily chosen, we get the desired estimate.
\end{proof}
\begin{remark}
     We can consider an alternative auxiliary subsolution to obtain the stability result, which is based on the idea of Fang~\cite{Fang25-continuity-hessian}. 
     We set $g_\delta=\mathbf{1}_{\{\varphi<\psi-\delta \}}f$. We see that $g_\delta\in L^q(dV_X)$ for any $1<q<p$ by the H\"older inequality.
It follows from Lemma~\ref{lem: subsolution} that there exists a constant $m>0$ and $v\in\PSH(X,\theta)$ such that $V_\theta-1\leq v\leq V_\theta$ and 
$$(\theta+\dc v)^n\geq c^n\frac{ g_\delta dV_X}{\|g_\delta\|_q},$$
for $c$ only depending on $\omega_X$, $p$, and $\theta$.
We proceed in the same way as above to obtain $$\sup_X(\psi-\varphi)\leq \delta+ C\|g_\delta\|_q^{1/n},$$
  for $C>0$ depending on $C_0$, $\omega_X$, $q$, $\theta$. 
H\"older's inequality yields 
\begin{align*}
    \|g_\delta\|_q\leq \frac{\|f\|_p\|(\psi-\varphi)_+\|_1^{\frac{p-q}{pq}}}{\delta^{\frac{p-q}{pq}}}
\end{align*}
Choosing $\delta:=\|(\psi-\varphi)_+\|_1^\gamma$ {with} $\gamma=\frac{p-q}{npq+p-q}$ we obtain
\[\sup_X(\psi-\varphi)\leq C\|(\psi-\varphi)_+\|^\gamma_1 \]
for $C$ depending on $C_0$, $p$, $q$ and $\|f\|_p$. Since $q\in (1,p)$ was taken arbitrarily, 
 the desired estimate holds for any 
$0<\gamma<\gamma_0:=\frac{p-1}{np+p-1}$.

\end{remark}
\subsection{H\"older continuity}
In this section, we study the H\"older continuity for solutions on the Zariski open set $\Omega:= \{\psi_0>-\infty\}$. In the K\"ahler case, the result is due to Demailly, Dinew, Guedj, Kolodziej, Pham, and Zeriahi~\cite[Theorem D]{demailly2014holder}. The main ingredient in the latter proof was the application of Demailly's regularization approximations for quasi-psh functions which makes use of the holomorphic part of the Riemannian exponential mapping~\cite{demailly1994regularization}. This method can also be applied to the non-K\"ahler case. We are going to follow the scheme of their proof with minor refinements.


\begin{proof}[Proof of Theorem~\ref{thm: local-Holder}]

Recall that $\psi_0$ is a $\theta$-psh function with analytic singularities such that $\theta+\dc\psi_0\geq \varepsilon_0\omega_X$ and $\sup_X\psi_0\leq 0$.
Set 
\[\f_{c,\delta}:= \frac{Ac+2K\delta}{\varepsilon_0}\psi_0+\left( 1-\frac{Ac+2K\delta}{\varepsilon_0}\right)\Phi_{c,\delta},\] where $0<\delta\leq\delta_0$, $c>0$ and $\Phi_{c,\delta}$ is the  Kiselman--Legendre transform; cf.~\eqref{eq: Kiselman}.
 In the following arguments, we choose $c=O(\delta^\gamma)$ so  that $Ac+2K\delta=\varepsilon_0\delta^{2\gamma}$ and write $\f_\delta$ instead of $\f_{c,\delta}$.  We observe that $\Phi_{c,\delta}\leq K\delta^2+K\delta$ so $\f_\delta-B\delta^{2\gamma}\leq 0$ for some $B>0$ and $0<\delta\leq \delta_0$ sufficiently small. By the stability result (Theorem~\ref{thm :stability-big}), we obtain
\begin{align*}
    \sup_X(\f_\delta-\f)&\leq C_0\|(\f_\delta-\f-B\delta^{2\gamma})_+\|^\gamma_{L^1}+B\delta^{2\gamma}\\
    &\leq C_0\|(\rho_\delta\f+K\delta-\f)\|^\gamma_{L^1}+B\delta^{2\gamma}.
\end{align*}
By~\cite[Lemma 2.3]{demailly2014holder} we have
\[\sup_X(\f_\delta-\f)\leq C_1\delta^{2\gamma},\]
where $C_1$ depends on $B$, $C_0$, $K$, $\|\f-V_\theta\|_{L^\infty}$ and the Chern curvature of $\omega_X$. 
 From this point we can conclude the H\"older continuity of $\f$ as in the proof of~\cite{demailly2014holder}. For the sake of completeness, we give all the details for the reader's convenience.

Fix the point $z\in \Omega$. 
Then the minimum in the definition of $\Phi_{c,\delta}$ is realized at $t_0=t_0(z)$. 
Therefore, the last inequality yields, at the point $z$,
\[\delta^{2\gamma}(\psi_0-\f)+(1-\delta^{2\gamma})(\rho_{t_0}\f+Kt_0-K\delta-\f-c\log(t_0/\delta))\leq C_1\delta^{2\gamma}. \]
Since $\rho_{t_0}\f+Kt_0-\f\geq 0$ we infer that
\[ c(1-\delta^{2\gamma})\log\frac{t_0}{\delta}\geq \delta^{2\gamma}(\psi_0(z)-V_\theta(z)-C_2).\]
Combining with $c=\varepsilon_0 A^{-1}\delta^{2\gamma}-2KA^{-1}\delta\geq\frac{1}{2} \varepsilon_0 A^{-1}\delta^{2\gamma}$ if $\delta\leq \delta_0$ sufficiently small, then one gets that
\[ t_0\geq \delta\kappa\;\; \text{with}\; \kappa(z)=\exp \left( \frac{2A(\psi_0(z)-V_\theta(z)-C_2)}{\varepsilon_0(1-\delta_0^{2\gamma})}\right).\] Since $t\mapsto\rho_t\f+Kt$ is increasing and $t_0(z)\geq \delta\kappa(z)$ it follows that
\begin{align*}
    \rho_{\delta\kappa}\f(z)+K\delta\kappa&\leq \rho_{t_0}\f(z)+Kt_0\\&\leq\Phi_{c,\delta}(z)= \frac{1}{1-\delta^{2\gamma}}(\f_\delta(z)-\delta^{2\gamma}\psi_0).
\end{align*}
By the stability result again, we obtain
\[ \f_\delta-\delta^{2\gamma}\psi_0\leq \f+\delta^{2\gamma}(C_2-\psi_0),\]
using that $\f\leq 0$, hence
\[\rho_{\delta\kappa(z)}\f(z)-\f(z)\leq \frac{\delta^{2\gamma}}{1-\delta^{2\gamma}} (C_2-\psi_0(z))\leq (1-\delta_0^{2\gamma})^{-1}\delta^{2\gamma}(C_2
-\psi_0(z)).\] Replacing $\delta$ by $\delta\kappa(z)^{-1}$, one gets
\[ \rho_\delta\f(z)-\f(z)\leq (1-\delta_0^{2\gamma})^{-1}(C_2-\psi_0(z))\exp \left( \frac{4A\gamma(C_2+V_\theta(z)-\psi_0(z))}{\varepsilon_0(1-\delta_0^{2\gamma})}\right)\delta^{2\gamma}. \]
It follows that for any compact set $K\subset\subset \Omega$ there exists a constant $C$ depending on $K$ such that for any $z\in K$,
\[\rho_\delta\f(z)-\f(z)\leq C\delta^{2\gamma}, \]
which implies the H\"older continuity of $\f$ on $K$; see e.g.,~\cite[Lemma 4.2]{guedj2008holder} or~\cite{Zeriahi20-continuity}. Since $K$ was taken arbitrarily, one can conclude that $\f$ is H\"older continuous on $\Omega$.     
\end{proof}

\begin{remark}
    Our argument in this section can also be adapted to establish the H\"older regularity of solutions to the complex Monge–Ampère equation in the local setting, as studied by Guedj--Kołodziej--Zeriahi~\cite{guedj2008holder} (see also~\cite{BKPZ16-holder,FengShiXu18-dirichlet,HongThuy18-holder,Nguyen20-Holder}). A key point in their approach is the stability result~\cite[Theorem 1.1]{guedj2008holder}. By following similar arguments to those presented above (see also Section~\ref{sect: global}), we can derive a corresponding stability estimate without using the complex Monge--Ampère capacity. 

    Our approach to the stability estimate can be applied to study the regularity of solutions to complex Hessian equations; see, e.g.,~\cite{KN16-hessian,GuNguyen16-hessian,MR4125759,WZ24-trace, Cheng-Xu24-m-subharmonic,Fang25-continuity-hessian}.
\end{remark}
\section{Global Continuity on Resolutions of Singularities}

\label{sect: global}
It is natural to ask whether the solution $\varphi$ to the equation~\eqref{eq: cmae-big} is continuous on the whole of $X$. In this section, we provide an affirmative answer to this question when $\theta$ is the pullback of a Hermitian form via birational and projective morphisms. 
\subsection{Continuity of potentials of Ricci-flat currents}
To state our result, we fix some notation and terminology. Let $(Y,\omega_Y)$ be a compact, locally irreducible, normal variety equipped with a Hermitian form $\omega_Y$. We assume that the canonical bundle $K_Y$ is $\mathbb{Q}$-Cartier  and  $Y$ has log terminal singularities.
Let $\pi:X\to Y$ be a resolution of singularities. We have\[\pi^*K_Y=K_X+\sum_i a_iE_i, \] where the $E_i$’s are exceptional divisors with simple normal crossings, and the rational coefficients $a_i$ (the discrepancies) satisfy $a_i>-1$.

Let $\sigma$ be a local non-vanishing holomorphic section of $K_Y^{\otimes r}$ and $h$ be a smooth metric of $K_Y$. We define the ``adapted volume form"
\[ \mu_{Y,h}:=\left( \frac{i^{rn^2}\sigma\wedge\Bar{\sigma}}{|\sigma|^2_{h^r}}\right)^{1/r}.\]
We note that this measure is independent of the choice of $\sigma$, and has finite mass on $Y$ since the singularities are log-terminal. 



Let $\omega_X$ be a smooth Hermitian form on $X$. 
We denote by $f$ the density of $\pi^*\mu_{Y,h}$ with respect to $\omega_X^n$. We observe that $f\in L^p$ since the singularities of $Y$ are log terminal; see, e.g.,~\cite{eyssidieux2009singular}. 
Set $\theta:=\pi^*\omega_Y$. It follows from \cite{guedj2021quasi} that there exist a unique constant $c>0$ and a bounded $\theta$-psh function $\varphi$ such that 
\begin{equation}\label{eq: cmae-X}
    (\theta+\dc\varphi)^n=cfdV_X, \qquad\sup_X\varphi=0.
\end{equation}

\begin{theorem}\label{thm: global-cont}
    If $\varphi\in\PSH(X,\theta)\cap L^\infty(X)$ is a solution to \eqref{eq: cmae-X}, then $\varphi$ is continuous on $X$.  
\end{theorem}

\begin{lemma}\label{lem: local-subsol} Let $\Omega$ be a relatively compact strictly
pseudoconvex domain in $X$. Let $\eta$ be a semipositive and big (1,1) form on $\overline{\Omega}$.
Let $g\in L^r(\Omega,dV_X)$ for some $r>1$ such that $\|g\|_r\leq C_r$ for some $C_r>0$. 
    There exists a constant $c=c(\Omega,\eta,dV_X,r,C_r)$, and $\psi\in\PSH(\Omega,\eta)$ such that $-1\leq \psi\leq 0$ and \[(\eta+\dc\psi)^n\geq c^n{g dV_X}\quad\text{in}\;\Omega. \]
\end{lemma}
\begin{proof} The proof is almost identical to that of~\cite[Lemma 2.1]{guedj2021quasi}.
Since $\eta$ is big, we can find an $\eta$-psh function $\chi$ on $\overline{\Omega}$ such that $\eta+\dc\chi\geq \varepsilon_0\omega_X$. We may assume that $\chi\leq -1$  on $\overline{\Omega}$. 
By the H\"older inequality, we infer that $h=(-\chi)^{2n}g\in L^{r'}(\Omega,dV_X)$ for some $1<r'<r$.  
    It follows from~\cite[Theorem 4.2]{kolodziej2015weak} that there exists a unique continuous solution $\phi\in\PSH(\Omega,\omega_X)$ to the equation
    \begin{equation*}
        \begin{cases}
            (\omega_X+\dc \phi)^n=(-\chi)^{2n}{gdV_X} \quad \text{in}\;\Omega,\\
            \phi|_{\partial\Omega}=-1.
        \end{cases}
    \end{equation*} We have $\|\phi\|_\infty\leq C\|g\|_{r'}^{1/n}$ where $C$ depends on $r'$ and $\omega_X$. We assume $\phi\leq 0$.
We set $\psi:=-\frac{1}{\chi+\varepsilon_0 \phi}-1$. 
Our normalizations
ensure $-1\leq \psi\leq 0$ and 
\begin{align*}
    \eta+\dc\psi&\geq (-\chi-\varepsilon_0\phi)^{-2}\varepsilon_0(\omega_X+\dc\phi),
\end{align*} since $\eta\geq 0$. It follows that $$(\eta+\dc \psi)^n\geq \varepsilon_0^n\frac{(-\chi)^{2n}}{(-\chi-\varepsilon_0\phi)^{2n}}{gdV_X}\geq c^n{gdV_X}, $$
for $c>0$ depending on $\varepsilon_0$, $n$ and $\|\phi\|_\infty$. 
\end{proof}
\begin{proposition}\label{prop: stability}
Let $\Omega$ be a relatively compact strictly
pseudoconvex domain in $X$. Let $\eta$ be a semipositive and big (1,1) form on $\overline{\Omega}$. Assume that $u,v\in\PSH(\Omega,\eta)\cap L^\infty(\Omega)$ such that $\liminf_{\partial\Omega}(u-v)\geq 0$ and 
\[(\eta+\dc u)^n=fdV_X\qquad\text{on}\;\Omega, \]
  where $f\in L^p(\Omega,dV_X)$ for some $p>1$. Let $\mu=fdV_X|_\Omega$. Fix any $q>1$.
  Then, there exists a constant $C=C(\Omega,\eta,dV_X,n,p,q,\|f\|_p,\|v\|_\infty)$ such that
  for any $\delta>0$,
  \[\sup_{\Omega}(v-u)\leq \delta+C[\mu{(\{u<v-\delta\})}]^{\frac{1}{nq}}. \]
\end{proposition}
\begin{proof} The proof is very close to that of Theorem~\ref{thm :stability-big}; we include it briefly for the reader's convenience.
   We may assume that $\mu(\{u<v-\delta\})>0$. 
    We set $$g_\delta=\frac{\mathbf{1}_{\{u<v-\delta\}}f}{[\mu(\{u<v-\delta\})]^{\frac{1}{q}}}.$$ We see that $g_\delta\in L^r(\Omega,dV_X)$ and $\|g_\delta\|_r\leq \|f\|_p^{1-\frac{1}{q}}$ with $\frac{1}{r}=\frac{1}{p}\big(1-\frac{1}{q}\big)+\frac{1}{q}$; following the same arguments as in the proof of Theorem~\ref{thm :stability-big}. 
    
    Let $\psi$ be an $\eta$-psh function defined in Lemma~\ref{lem: local-subsol} for $g_\delta$ and $r>1$.
    We set $\varphi:=(1-\lambda)v+\lambda\psi$, 
    where $\lambda\in (0,1)$ is chosen later. 
   There is a constant $C=C(\|v\|_\infty)>0$ such that $$u\geq \varphi-\delta-C\lambda\; \text{on}\; \partial\Omega,$$ and \[ \{u<\varphi-\delta-C\lambda \}\subset\{ u<v-\delta\}.\]
   We compute
    \[\mathbf{1}_{\{u<\varphi-\delta-C\lambda \}}(\eta+\dc \varphi)^n\geq\lambda^n\mathbf{1}_{\{u<\varphi-\delta-C\lambda \}}(\eta+\dc \psi)^n\geq\lambda^nc^n{g_\delta dV_X}. \]
    If $[\mu(\{u<v-\delta\})]^{\frac{1}{nq}}\geq c$, we are done. Otherwise, if we choose $$\lambda=c^{-1}[\mu(\{u<v-\delta\})]^{\frac{1}{nq}}\in (0,1),$$ 
    then 
    \[\mathbf{1}_{\{ u<\varphi-\delta-C\lambda\}}(\eta+\dc \varphi)^n\geq \mathbf{1}_{\{ u<\varphi-\delta-C\lambda\}}fdV_X|_\Omega=\mathbf{1}_{\{ u<\varphi-\delta-C\lambda\}}(\eta+\dc u)^n. \]
    By~\cite[Lemma 4.1]{BoucksomGuedjLu2025-volume}, we get $u\geq \varphi-\delta-C\lambda$ on $\Omega$. This implies the desired estimate.
\end{proof}
{ 
\begin{lemma}\label{lem: approx}
Let $B$ be a Stein complex space endowed with a Hermitian form
$\omega$, and let
\(
    \phi\in \PSH(B,\omega)\cap L^\infty_{\rm loc}(B)
\).
Then, for every relatively compact open subset $B'\Subset B$, there
exists a sequence
\(
    \phi_j\in \PSH(B',\omega)\cap\mathcal C^\infty(B')
\)
such that
\(
    \phi_j\searrow\phi\)
    {on} \(B'\).
\end{lemma}

\begin{proof}
Fix $B'\Subset B$ and choose an open subset $B''$ such that
\[
    B'\Subset B''\Subset B.
\]
After subtracting a constant, we may assume that $\phi\leq -1$ on
$B''$.
 Let $\omega_0$ be a Hermitian form on $B''$. 
By \cite[Theorem 5.5]{fornaess1980levi} and~\cite{blocki2007regularization}, there exist a sequence $\varepsilon_j\searrow0$ and smooth functions
\[
    \psi_j\in
    \PSH(B',\omega+\varepsilon_j\omega_0)
    \cap\mathcal C^\infty(B')
\]
such that
\(
    \psi_j\searrow\phi\)
    on $B'$.
Since $\phi\leq-1$, we may assume
that $\psi_j<0$ on $B'$.

Since $\overline{B'}$ is compact, there exists a constant $C>0$ such that
\(
    \omega_0\leq C\omega\) on $B'$.
Set
\(
    \phi_j:=\frac{\psi_j}{1+C\varepsilon_j}.
\)
Then $ \phi_j\in\PSH(B',\omega)\cap\mathcal C^\infty(B')$.
Moreover, since $\psi_{j+1}\leq\psi_j<0$ and
$\varepsilon_{j+1}\leq\varepsilon_j$, we have
\[
    \phi_{j+1}
    =\frac{\psi_{j+1}}{1+C\varepsilon_{j+1}}
    \leq
    \frac{\psi_j}{1+C\varepsilon_{j+1}}
    \leq
    \frac{\psi_j}{1+C\varepsilon_j}
    =\phi_j.
\]
Thus $(\phi_j)$ is decreasing. Finally, since
$\psi_j\to\phi$ and $\varepsilon_j\to0$, $\phi_j\searrow\phi$ on $B'$.
\end{proof}
 }

\begin{proof}[Proof of Theorem~\ref{thm: global-cont}] Denote $\varphi_*$ the lower semicontinuous regularization of $\varphi$.
Suppose by contradiction that $\varphi$ is not continuous on $X$, then we have \[d:=\sup_X(\varphi-\varphi_*)>0. \]
Set $F:=\{x\in X: \varphi(x)-\varphi_*(x)=d \}$. We observe that $F$ is a closed nonempty set since $\varphi-\varphi_*$ is a bounded upper semi-continuous function on the compact set $X$. 
On $F$ we have $\varphi=\varphi_*+d$. Since $\varphi$ is upper semicontinuous and $\varphi_*$ is lower semicontinuous, $\varphi|_F$
is continuous.
Thus, we can pick $x_0\in F$ such that \[\varphi(x_0)=\min_F\varphi. \]
Set $y_0=\pi(x_0)$.
It follows from ~\cite[Lemma 3.1]{ChoChoi25-continuity} that we can choose an open Stein neighborhood $B$ of $y_0$ that admits a local embedding
$
\iota:B\hookrightarrow \mathbb C^N.
$ and $\rho\in\mathcal{C}^\infty(B)$ such that $\omega_Y> \dc \rho>0$ on $B^{\rm reg}$ and \[\rho(y_0)<\rho(y) \;\text{for each}\; y\in B\backslash\{y_0\}.\] After shrinking \(B\), we may assume that
$
\iota(B)
$ is an intersection
of a complex analytic space and some Euclidean ball in \(\mathbb C^N\). 
Moreover, each component of the preimage of $B$ is birational to it. We choose one $\widetilde{B}$, which contains $x_0$. 
We define the pushforward $\pi_*\varphi$ of $\varphi$ on $\widetilde B$ as follows
\[\pi_*\varphi(y)=\begin{cases}
    \varphi(x), y\in Y^{\rm reg},\pi(x)=y;\\
    \limsup_{\zeta\in \widetilde{B}\backslash E,\pi(\zeta)\to y}\varphi(\zeta).
\end{cases} \] We observe that $\phi:=\pi_*\varphi\in \PSH(B,\omega_Y)\cap L^\infty(B)$. 















One can find a relatively compact open neighborhood $U\Subset B$ of $x_0$ on which $y_0$ is a strict minimum point of $\rho$.
Set $S:=\partial U$. Since $\pi(S)$ is compact and
$y_0\notin\pi(S)$, while $\rho$ has a strict minimum at $y_0$, we have
\[
    b:=\inf_S\bigl(\rho\circ\pi-\rho(y_0)\bigr)>0.
\]
Set $u:=\rho\circ\pi+\varphi$ and $\eta:=\theta-\dc \rho\circ \pi$. We see that $\eta$ is semipositive and big (see, e.g.,~\cite[Proposition 3.2]{fino2009blow}) and that $u\in\PSH(U,\eta)$ solves the equation
\[(\eta+\dc u)^n= cfdV_X \quad\text{on}\; U.\]
We also have the following \[\sup_{\Bar{U}}(u-u_*)=\sup_{\Bar{U}}(\varphi-\varphi_*)=d. \]
For each $a\in [0,d]$, we consider
 \[E(a)=\{z\in \Bar{U}:\varphi(z)-\varphi_*(z)\geq d-a\} \ni z_0,\] and \[c(a):=\varphi(x_0)-\inf_{E(a)}\varphi. \] 
We observe that $E(a) $ is a compact set for any $a\in [0,d]$ such that $E(a)\searrow E(0)$ and $c(a)$ is a nondecreasing function. 
\medskip

\noindent\textbf{Claim 1.} $\lim_{a\to 0}c(a)=0$.

Since $\liminf_{a\to 0} c(a)\geq 0$, it suffices to show that $\limsup_{a\to 0} c(a)\leq 0$. Suppose by contradiction that  \[ \limsup_{a\to 0} c(a)> 2\varepsilon,\] for some $\varepsilon>0$. Then, there exists a sequence $a_j>0$ such that $c(a_j)>\varepsilon$ for every $j\geq 1$, namely,
\[\inf_{E(a_j)}\varphi<\varphi(x_0)-\varepsilon. \]

By the definition of infimum, there exists a point $x_j\in E(a_j)$ such that $\varphi(x_j)<\varphi(x_0)-\varepsilon$. Let $x$ be a limit point of $x_j$. Then   
\[\varphi(x)\geq \liminf_j\varphi(x_j)\geq \liminf_j(\varphi_*(x_j)+d-a_j)\geq \varphi_*(x)+d,  \]
hence $x\in E(0)$. Hence,
\[ \limsup_j \varphi(x_j)\leq \varphi(x_0)-\varepsilon.\]
By the upper semicontinuity of $-\varphi_*$, we have the following
\[ d=\limsup_j[ \varphi(x_j)-\varphi_*(x_j)]\leq  \varphi(x_0)-\varepsilon-\varphi_*(x)\leq d-\varepsilon.\] Since $\varphi_*(x)=\varphi(x)-d\geq \varphi(x_0)-d= \varphi_*(x_0)$, we get a contradiction. 

\medskip
By Lemma~\ref{lem: approx}, after shrinking $B$ if necessary,
there exists an open set $B'\Subset B$ with
$\pi(U)\subset B'$ and a sequence of smooth functions
\(
\{\phi_j\}\subset \PSH(B',\omega_Y)
\)
that decreases to $\phi=\pi_*\varphi$. Set
 \(
 u_j:=(\rho+\phi_j)\circ\pi\in\PSH(U,\eta)
 \). Then $u_j\searrow u$ on $U$.

\smallskip The following claim is a variation of Hartogs' lemma.

 \noindent\textbf{Claim 2.} Let $c>0$, $t>1$. If $u-tu_*< c$ on  a compact set $K\subset U$, then there exists a positive integer $j_0=j_0(K,t)$ such that near $K$, \[u_j\leq tu_*+c, \quad\text{for}\, j\geq j_0.\]

 \smallskip 
The proof of the claim is almost identical to that of~\cite[Corollary 2.9]{ChoChoi25-continuity}, so we omit the details here.

\medskip 
Up to adding a positive constant, we can assume that $u> 0$ on $U$ and $A:=u(x_0)>d$. 
By Claim 1, we can choose $a_0\in (0,d)$ such that \[ c(a)<\frac{b}{3}\] for any $a\leq a_0$. In addition, we choose $t>1$ so that
\begin{equation}\label{eq: a0t}
    (t-1)(A-d)< a_0<(t-1)\left(A-d+\frac{2b}{3}\right).
\end{equation} 
\textbf{Claim 3.} There exists an open neighborhood $V$ of $S=\partial U$  and  $j_0\in\mathbb Z_{>0}$ such that \[ u_j<tu+d-a_0\; \text{for}\; j\geq j_0. \]

\textit{Case 1.} $z\in S\cap E(a_0)$. It follows that
\begin{align*}
    u_*(z)&\geq \rho(\pi(x_0))+b+ \varphi(z)-d\\&\geq\rho(\pi(x_0))+b+\varphi(x_0)-c(a_0)-d \\&\geq A-d+\frac{2b}{3}.
\end{align*} 
Hence, by \eqref{eq: a0t}, $(t-1)u_*(z)>a_0$, it follows that  \[ u(z)\leq u_*(z)+d<tu_*+d-a_0.\] 
Hartogs' lemma (Claim 2) yields the result that there exists an open neighborhood $V_1$ of the compact set $S\cap E(a_0)$
and $j_1\in\mathbb{Z}_{>0}$ such that
 \[ u_j<tu_*+d-a_0\] on $V_1$, for $j\geq j_1$.

\textit{Case 2.} $E(a_0)\cap (S\setminus V_1)=\varnothing$. We have 
\[ u<u_*+d-a_0\] on $S\setminus V_1$. We apply Hartogs' lemma again to obtain that there exists a neighborhood $V_2$ of $S\setminus V_1$  and $j_2\in\mathbb Z_{>0}$ such that \[ u_j<u_*+d-a_0<tu_*+d-a_0\] on $V_2$, for $j\geq j_2$. Therefore, setting $V=V_1\cup V_2$ and $j_0=\max(j_1,j_2)$, we obtain the claim.

\medskip

Now, we are ready to obtain a contradiction.
Set $w=tu+d-a_0$ and for $c>0$
\[ W(j,c):=\{x\in \Bar U: w(x)+c<u_j(x) \}.\]
By~\eqref{eq: a0t}, we have \[a_0>(t-1)(A-d)=(t-1)u_*(x_0)=tu_*(x_0)+d-u(x_0), \]
hence, there exists a positive constant $a_1>0$ so small that
\[  tu_*(x_0)+d<a_0+u(x_0)-2a_1<a_0+u_j(x_0)-a_1,\] for $j$ sufficiently large.
This implies that the sets $W(j,c) $ for $c\in (0,a_1)$ contain some points near $x_0$, and so they are non-empty. We also see that by Claim 3, $W(j,c)$ is relatively compact in $U$ for $j\geq j_0$. It follows that \[\sup_U(u_j-w)> \frac{a_1}{2},\quad\text{for}\; j\geq j_0. \]  
We apply Proposition~\ref{prop: stability} (with $v=u_j$, $u=w\in\PSH(U,t\eta)$ and $\delta=\frac{a_1}{4}$) to obtain  \[\frac{a_1}{4}\leq Ct\left[\mu\left(\left\{w<u_j-\frac{a_1}{4}\right\}\right)\right]^{\frac{1}{nq}}, \] for a uniform constant $C>0$.  Note that $(t\eta+\dc w)^n=t^ncfdV$.
Furthermore, for such a fixed $a_1>0$ we have
\[W\left(j,\frac{a_1}{4}\right)\subset \left\{ x\in \Bar U: u(x)+d-a_0+\frac{a_1}{4}<u_j(x) \right\}. \]
Since by the dominated convergence theorem, $\mu(\{u+d-a_0+\frac{a_1}{4}<u_j\})\to 0$ as $j\to+\infty$, we get the contradiction. This completes the proof.
\end{proof}
\subsection{Continuity of potentials of pluripotential Chern–-Ricci flows.} We apply the continuity result for solutions to the elliptic Monge--Amp\`ere equation to prove the continuity of solutions to the parabolic one. 

As in the previous section, we assume that $Y$ is a compact Hermitian variety with log-terminal singularities. Let $\pi:X\to Y$ be a log resolution of singularities.  

We consider  the following  parabolic complex Monge–Ampère type equation
	\begin{align*}\label{cmaf} \tag{CMAF}
	dt\wedge(\omega_t+\dc\f_t)^n=e^{\partial_t{\f}_t+F(t,x,\f_t)}f(x)dV_X(x)\wedge dt
	\end{align*}
	in $X_T:=(0,T)\times X$, where 
	\begin{itemize}
		\item $(\theta_t)_{t\in [0,T]}$ is a smooth family of  (1,1) forms such that $ \theta_t=\pi^*\omega_t$ where $\omega_t$ is a family of Hermitian forms on $Y$, and
		\begin{equation}\label{assum: omega}
	    -A\theta_t\leq \partial_t{\theta}_t\leq A\theta_t\;\, \text{and}\;\, \partial^2_{tt}{\theta}_t\leq A\theta_t,\; \forall \, t\in [0,T],
	\end{equation} for some fixed constant $A>0$;
    
		
		\item $(t,x,r)\mapsto F(t,x,r)$ is continuous on $[0,T]\times X\times \mathbb{R}$, quasi-increasing in $r$, uniformly Lipschitz in $(t,r)$, and uniformly convex in $(t,r)$;
		
		\item $0\leq f\in L^p(X)$ for some $p>1$, and $f>0$ almost everywhere.
		
		 \item $\f:[0,T]\times X\rightarrow \mathbb{R}$ is an unknown function, with $\f_t:=\f(t,\cdot)$. 
	\end{itemize}
    When $F\equiv 0$ and $\omega_t=\omega_X+t\chi+\dc\rho_t>0$ for $\rho_t\in\mathcal{C}^\infty(X)$, then the flow \eqref{cmaf} is called the {\em weak Chern--Ricci} flow.
    Generalizing the work of Guedj--Lu--Zeriahi~\cite{guedj2020pluripotential}, we proved in~\cite{dang24-chern} the existence and uniqueness of this flow.
\begin{theorem}
    { Let $\f_0$ be a bounded $\omega_0$-psh function. Then there exists a unique solution $\f_t\in\PSH(X,\omega_t)\cap L^\infty(X)$ to \eqref{cmaf} such that $\f_t\to\f_0$ as $t\to 0^+$ in $L^1(X)$ and for all $0<T'<T$,
\begin{itemize}
     \item $(t,x)\mapsto\f(t,x)$ is bounded in $[0,T')\times X$,
    \item $t\mapsto\f_t$ is uniformly semi-concave in $(0,T')\times X$,
    \item \[n\log t-C\leq \partial_t\varphi\leq \frac{C}{t},\quad \text{on}\; (0,T')\times X \]
    for a uniform $C>0$.
\end{itemize}}
\end{theorem}
It follows from ~\cite[Proposition 3.6]{dang24-chern} that $\varphi$ is continuous on $(0,T)\times \Omega$ for some open Zariski set $\Omega\subset X$. The following result improves our previous result.
\begin{theorem} Let $\f_0$ be a bounded $\omega_0$-psh function.
    The potential $\varphi:[0,T)\times X\to \mathbb R$ of the unique solution to the flow \eqref{cmaf} on $X$ starting from $\omega_0+\dc\varphi_0$ is continuous on $(0,T)\times X$. In particular, if $\varphi_0$ is continuous on $X$, then $\varphi$ is continuous on $[0,T)\times X$. 
\end{theorem}
{\begin{proof}
Fix $J\Subset (0,T)$. Since $\varphi$ is locally uniformly Lipschitz
in $t$ and $F$ is bounded from above, there exists a constant
$M=M(J)>0$ such that, for almost every $t\in J$,
\[
    (\omega_t+dd^c\varphi_t)^n\leq e^M f\,dV_X.
\]
By Theorem~\ref{thm: global-cont}, $\varphi_t$ is continuous on $X$
for almost every $t\in J$.
Let $\kappa>0$ be a uniform Lipschitz constant of $\varphi$ with
respect to $t$ on $J$. 
Fix $t\in J$ and choose a sequence
$t_j\to t$ such that $\varphi_{t_j}$ is continuous on $X$. Then
\[
    \|\varphi_{t_j}-\varphi_t\|_{L^\infty(X)}
    \leq \kappa |t_j-t|\longrightarrow 0.
\]
Thus, $\varphi_t$ is continuous on $X$.
Fix $(t,x)\in J\times X$ and let
$(t_j,x_j)\to(t,x)$. Then
\[
\begin{aligned}
|\varphi(t_j,x_j)-\varphi(t,x)|
&\leq
|\varphi(t_j,x_j)-\varphi(t,x_j)|
 +|\varphi(t,x_j)-\varphi(t,x)|\\
&\leq
\kappa|t_j-t|
 +|\varphi(t,x_j)-\varphi(t,x)|.
\end{aligned}
\]
The right-hand side tends to zero since $\varphi_t$ is continuous on
$X$. Hence $\varphi$ is continuous on $J\times X$. Since
$J\Subset(0,T)$ is arbitrary, $\varphi$ is continuous on
$(0,T)\times X$.
\end{proof}
}

\section{Geometric applications}\label{sect: application}
\subsection{Modulus of continuity and diameter bound}

 
Let $\mu=fdV_X$ be a positive Radon measure with $f\in L^p(dV_X)$ for some $p>1$. It follows from ~\cite{dinew2012pluripotential,kolodziej2015weak} that there exists a continuous $\omega_X$-psh function $\varphi$ and $c>0$ such that $$(\omega_X+\dc\varphi)^n=cfdV_X,\quad\sup_X\varphi=0,$$ with a uniform a priori bound on $c$ and $\|\varphi\|_\infty$.

\begin{theorem}\label{thm: stability-L1} Let $\mu$ be a positive Radon measure with $f\in L^p(dV_X)$ for some $p>1$.
    Let $\psi$ be a bounded $\omega_X$-psh function with $\sup_X\psi\leq 0$. Assume $\varphi\in \PSH(X,\omega_X)\cap {L}^\infty(X)$ and $c>0$ solve the complex Monge--Amp\`ere equation
    \[(\omega_X+\dc\varphi)^n=cfdV_X,\quad\sup_X\varphi=0. \] Fix $0<\alpha<\frac{p-1}{np+p-1}.$
    Then
    \begin{equation}\label{eq: stable}
        \sup_X(\psi-\varphi)\leq C (\|(\psi-\varphi)_+\|_{1})^\alpha,
    \end{equation}
  where $C>0$ depends on $n$, $\omega_X$,  $\alpha$, $p$ and $\|f\|_{p}$.
    
     Moreover, the modulus of continuity of $\varphi$ satisfies $m_\varphi(r)=O(r^{2\alpha})$; in particular, $\varphi$ is H\"older continuous with exponent $\alpha$ on $X$.
\end{theorem} 
Stability estimates were established by Ko\l odziej and Nguyen
in~\cite{kolodziej2018holder}, where they developed the Monge--Amp\`ere capacity
in the Hermitian setting. Ko\l odziej and Nguyen~\cite{kolodziej2019stability} proved the
H\"older continuity of solutions to the complex Monge--Amp\`ere
equation under the additional assumption that the right-hand side
is strictly positive. This assumption was later removed
in~\cite{kolodziej2018holder}, and the H\"older exponent was subsequently sharpened
by Lu--Phung--T\^o~\cite{lu2021stability}.

 \begin{proof}

The stability estimate~\eqref{eq: stable} is a consequence of Theorem~\ref{thm :stability-big} where $\theta=\omega_X$.
 
We now prove the modulus of continuity. Recall that $\Phi_{c,\delta}$ is the Kiselman--Legendre transform of $\varphi$, defined in \eqref{eq: Kiselman}. 
We consider 
\[\psi_{\delta}=(1-[Ac+2K\delta])\Phi_{c(\delta),\delta}, \]
which is $\omega_X$-psh. We observe that
\[\psi_\delta\leq \Phi_{c,\delta}\leq \rho_\delta\varphi. \]
From \cite[Lemma 2.3]{demailly2014holder}, $\|\rho_\delta\varphi-\varphi\|_1=O(\delta^2)$, hence $$\|(\psi_\delta-\varphi)_+\|_1\leq \|\rho_\delta\varphi-\varphi\|_1=O(\delta^2).$$
We apply the stability theorem (Theorem~\ref{thm: stability-L1}) with $\psi=\psi_\delta-B\delta^{2\alpha}$ to obtain
$$\psi_\delta\leq\varphi +C\delta^{2\alpha},$$
with $C>0$ under control.  At $z\in X$, we can find $t=t(z)$ achieving the infimum of $\Phi_{c,\delta}$. 
Since $\rho_t\varphi+Kt\geq \varphi$ we have
\begin{align*}
    \varphi(z)+C\delta^{2\alpha}&\geq \rho_t\varphi(z)+Kt-K\delta-c\log\frac{t}{\delta}\\
    &\geq \varphi(z)-K\delta-c(\delta)\log\frac{t}{\delta}.
\end{align*}
If we choose $c(\delta)=K\delta+C\delta^{2\alpha}=O(\delta^{2\alpha})$, then $\log\frac{t}{\delta}\geq -C$. Hence $t\geq \kappa \delta$  for some $\kappa\in (0,1)$. This yields at any $z\in X$ and $\delta\in[0,\delta_0)$ we have
$$\rho_{\kappa\delta}\varphi+K\kappa\delta-K\delta-\varphi\leq \rho_{t}\varphi+Kt-K\delta-\varphi=\Phi_{c,\delta}-\varphi.$$ We apply the stability estimate again to have that
$$\rho_{\kappa\delta}\varphi-\varphi\leq K\kappa\delta+C\delta^{2\alpha}\leq C'\delta^{2\alpha}.$$ Replacing $\delta$ by $\kappa^{-1}\delta$ we have
$$\rho_\delta\varphi-\varphi\leq C''\delta^{2\alpha}.$$
This yields the H\"older continuity of $\varphi$, see, e.g., \cite[Theorem 3.4]{Zeriahi20-continuity} or~\cite[Lemma 4.4]{lu2021stability}.\end{proof}
\begin{proposition}\label{prop: diameter}
    Let $(X,\omega_X)$ be a compact Hermitian manifold of complex dimension $n$. Assume $\varphi\in\PSH(X,\omega_X)$ is continuous in an open set $U\subset X$, with modulus of continuity $m_\varphi$ that satisfies the Dini condition $m_1(r):=\int_0^r\frac{\sqrt{m_\varphi(2t)}}{t}dt<+\infty$ and $m_\varphi(r)\geq r^2$ for $r>0$ sufficiently small. If $\omega:=\omega_X+\dc\varphi$ is a Hermitian form in $U$, then for each
compact set $K\subset U$ there exists $C_K>0$ such that for all $p,q\in K$ satisfying $d_{\omega_X}(p,q)<r_K$,
$$d_\omega(p, q) \leq  C_K m_1\circ d_{\omega_X} (p, q).$$
\end{proposition}
Here, $d_\omega$ denotes the intrinsic distance induced by $\omega$.
\begin{proof}
The idea goes back to Y. Li \cite{Li_yang_2021}, but requires slight modifications in the Hermitian context. 
We give a proof here for the reader's convenience. 

 Shrinking a positive constant $r_K$ if necessary, we
may assume that
\(
B_{\omega_X}(p,2r)\Subset U
\)
for every $p\in K$ and every $0<r<r_K$. We may also assume that $r_K<1$ and that the usual local volume comparison and
Poincar\'e--Wirtinger inequalities for $\omega_X$ hold uniformly on these balls.
Fix $p\in K$ and set
\(
d_p(x):=d_\omega(p,x).
\)
The function $d_p$ is $1$-Lipschitz with respect to $d_\omega$.
Consequently, it is differentiable almost everywhere and
\(
|\nabla d_p|_\omega\le 1
\)
almost everywhere on $U$.
For $0<r<r_K$, choose a smooth cutoff function $\chi$ such that
\[
0\le\chi\le1,\qquad
\chi\equiv1\quad\text{on }B_{\omega_X}(p,r),
\qquad
\operatorname{supp}\chi
\subset B_{\omega_X}(p,2r),
\]
and
\begin{equation}\label{eq:good_cutoff}
|\nabla\chi|_{\omega_X}\le Cr^{-1},
\qquad
|\nabla^2\chi|_{\omega_X}\le Cr^{-2},
\end{equation}
where $C>0$ is independent of $p$ and $r$.
We compute
\begin{align*}
    \int_{B_{\om_X}(p,r)} |\nabla d_p|_{\om_X}^2 \om_X^n
    &\leq \int_{B_{\om_X}(p,r)} |\nabla d_p|_{\om}^2 (\tr_{\om_X} \om) \om_X^n\\
    &\leq \int_{X} \chi |\nabla d_p|_{\om}^2 (\tr_{\om_X} \om) \om_X^n\\
    &\leq  \int_{X} \chi (\tr_{\om_X} \om) \om_X^n\\
    &= \int_{X} n \chi \om_X^n + \underbrace{\int_{X} \chi (\tr_{\om_X} dd^c \varphi) \om_X^n}_{=: I}.
\end{align*}
Since $\chi$ is compactly supported in $U$, integration by parts gives
\[
I
=
n\int_X
(\varphi-\varphi(p))
\,dd^c(\chi\omega_X^{n-1}).
\]
Notice that, unlike in the K\"ahler case, the form
$dd^c(\chi\omega_X^{n-1})$ contains torsion terms. More precisely,
up to the convention for $d^c$,
\begin{align*}
dd^c(\chi\omega_X^{n-1})
={}&
dd^c\chi\wedge\omega_X^{n-1}
-d^c\chi\wedge d\omega_X^{n-1}\\
&+
d\chi\wedge d^c\omega_X^{n-1}
+\chi\,dd^c\omega_X^{n-1}.
\end{align*}
Since $\omega_X$ is fixed and smooth, \eqref{eq:good_cutoff} yields
\[
\bigl|
dd^c(\chi\omega_X^{n-1})
\bigr|_{\omega_X}
\le
C(r^{-2}+r^{-1}+1)
\le Cr^{-2}.
\]
For $x\in B_{\omega_X}(p,2r)$ we have
\(
|\varphi(x)-\varphi(p)|
\le m_\varphi(2r).
\)
It follows that
\begin{align*}
|I|
&\le
C m_\varphi(2r)r^{-2}
\Vol_{\omega_X}\bigl(B_{\omega_X}(p,2r)\bigr)\\
&\le
C m_\varphi(2r)r^{2n-2}.
\end{align*}
Therefore, we obtain
\begin{equation}\label{eq:energy-dp}
\int_{B_{\omega_X}(p,r)}
|\nabla d_p|_{\omega_X}^2\,\omega_X^n
\le
C\left(
r^{2n}
+
m_\varphi(2r)r^{2n-2}
\right).
\end{equation}
By the Poincar\'e--Wirtinger inequality, we have
\begin{align*}
  \int_{B_{\om_X}(p,r)}\left|d_p-\underline{d_p} \right|^2\omega_X^n&\leq C_pr^2\int_{B_{\om_X}(p,r)}|\nabla d_p|^2_{\omega_X}\omega_X^n\\
  &\leq Cr^{2n}
\bigl(r^2+m_\varphi(2r)\bigr),
\end{align*} where $\underline{d_p}:= \frac{1}{\Vol(B_{\om_X}(p,r))}\int_{B_{\om_X}(p,r)}d_p\omega_X^n$.
Thus $d_p$ belongs locally, uniformly in $p\in K$, to the generalized
Morrey--Campanato space associated with the modulus $m_\varphi(2r)$ since $m_\varphi(r)\geq r^2$.
By the generalized Campanato--Morrey estimate
(see, e.g.,~\cite{Kovats99-elliptic}), $d_p$ admits a continuous
representative, for $x,y$ sufficiently close,
\[
|d_p(x)-d_p(y)|
\le
C_K
\int_0^{d_{\omega_X}(x,y)}
\frac{\sqrt{m_\varphi(2t)}}{t}\,dt.
\]
Taking $x=p$ and $y=q$ and using $d_p(p)=0$ give

\[
d_\omega(p,q)
\le
C_Km_1\bigl(d_{\omega_X}(p,q)\bigr)
\]
for $p,q$ sufficiently close.
Finally, since $K$ is compact, the local estimate can be extended to
pairs of points in the same connected component under consideration
by increasing $C_K$ if necessary. This completes the proof.
\end{proof}

 \subsection{Estimates along Chern--Ricci flows}
In this section, we study a generalization of the Chern–-Ricci flow, namely the {\em twisted Chern–-Ricci flow}: a smooth family of Hermitian metrics $(\omega_t)_{t>0}$ satisfies 
 \begin{equation}\label{eq: crf}
     \frac{\partial\omega_t}{\partial t}=-\textrm{Ric}(\omega_t)+\eta ,\quad \omega_t\xrightarrow{t\to 0}T_0\;\text{weakly},
 \end{equation}
 where $\textrm{Ric}(\omega_t)$ is the Chern-–Ricci form of $\omega_t$, $\eta$ is a smooth (1, 1) form and
 $T_0=\omega_X+\dc\varphi_0$ is a positive (1,1) current. Solving the twisted Chern--Ricci flow is equivalent to solving the parabolic Monge--Amp\`ere equation 
 \begin{equation}\label{eq: cmaf}
     (\widehat\omega_t+\dc\varphi_t)^n=e^{\partial_t\varphi}\omega_X^n,\quad\varphi_t\xrightarrow{t\to 0}\varphi_0,\;\text{in}\,L^1(\omega_X^n),
 \end{equation} where $\widehat\omega_t:=\omega_X+t(\eta-\textrm{Ric}(\omega_X))$. 
 In this case, $\omega_t=\widehat\omega_t+\dc\varphi_t$.

 Generalizing the work of Tosatti--Weinkove~\cite{tosatti2015evolution}, T\^o~\cite{to2018regularizing} showed that if $\varphi_0$ is a bounded $\omega_X$-psh function, then 
 there exists a unique solution of the twisted Chern--Ricci flow~\eqref{eq: crf}  on $[0,T_{\max})$, where
 \[ T_{\max}=\sup\{t>0:\exists\,\hat{\psi}_t\in\mathcal{C}^\infty(X),\widehat\omega_t+\dc\hat\psi_t>0 \}.\]
 More precisely, we have
 \begin{theorem}
     Let $\varphi_0$ be a continuous $\omega_X$-psh function. Then, there exists a unique solution $\varphi\in\mathcal{C}^\infty((0,T_{\max})\times X )$ to~\eqref{eq: cmaf} such that $\|\varphi_t-\varphi_0\|_{L^\infty}\to 0$ as $t\to 0^+$.\end{theorem}
As in~\cite{deruelle2025k}, we consider the case where $\eta=\textrm{Ric}(\omega_X)$, then $T_{\max}=+\infty$ and equation~\eqref{eq: cmaf} becomes 
\begin{equation}
     (\omega_X+\dc\varphi_t)^n=e^{\partial_t\varphi}\omega_X^n,\quad\varphi_t\xrightarrow{t\to 0}\varphi_0,\;\text{in}\,L^1(\omega_X^n).
     \end{equation}
\begin{remark} In general, since we are interested in the behavior of the flow near $0$, we can assume that for $0<S<T_{\max}$, there is a constant $C_S>0$ such that 
\[ C_S^{-1}\omega_X\leq \widehat\omega_t\leq C_S\omega_X,\; \, \forall\, t\in[0,S].\]
   The uniform bounds in the sequel are independent of $t$.  
\end{remark}

     In the remainder os this subsection, we assume $\eta=\textrm{Ric}(\omega_X)$.
     We further assume that $\varphi_0$ is continuous and $T_0^n=e^{\psi^+-\psi^-}\omega_X^n$ where $\psi^{\pm}$ are quasi-psh functions on $X$ with $e^{-\psi^-}\in L^p(\omega_X^n)$, for $p>1$.
     \begin{proposition}
    \label{prop: estimates}
    Fix $0<S<T_{\max}$.
         There exists a uniform constant $C>0$ such that for all $0\leq t\leq S$ and $x\in X$, 
         \begin{enumerate}
             \item $-C\leq \varphi\leq C$;
             \item $-C+\psi^+\leq  \partial_t\varphi_t\leq C-\psi^-$;
             \item $\omega_X+\dc\varphi_t\leq Ce^{-\delta\psi^-}\omega_X$, for some $\delta>0$.
         \end{enumerate}
     \end{proposition}
           \begin{proof} We choose $A>1$ so that $\psi^{\pm}\in\PSH(X,A\omega_X)$. We normalize $\sup_X\psi^{\pm}\leq 0$. We let $\Delta_t$ denote the Laplacian with respect to $\omega_t=\omega_X+\dc\varphi_t$.
           
\textbf{Case 1:} $\varphi_0$ and $\psi^{\pm}$ are smooth. 
  It follows immediately from the maximum principle that $\inf_X\varphi_0\leq \varphi_t(x)\leq \sup_X\varphi_0$ for all $x\in X$, $t>0$. 

 We prove (2). 
Consider for $(t,x)\in [0,1]\times X$,
 \[ H(t,x)=\partial_t\varphi_t(x)+\psi^-(x)-A\varphi_t(x).\]
 If $H$ achieves its maximum at $t=0$ we obtain   
 \[ H\leq AC.\]
Otherwise, $H$ achieves its maximum at $(t_0,x_0)$ with $t_0>0$.  At $(t_0,x_0)$, we have
\[0\leq \left(\frac{\partial}{\partial t}-\Delta_t \right)H=A(n-\partial_t\varphi_t)-\tr_{\omega_t}(A\omega_X+\dc\psi^-)\leq A(n-\partial_t\varphi_t),  \]
 hence $\partial_{t}\varphi_{t_0}(x_0)\leq n$. This implies that
 \[H\leq n+AC. \]
Since $|\varphi_t(x)|\leq C$, we get the upper bound for $\partial_t\varphi_t$.

For the lower bound of $\partial_t\varphi_t$, we consider
\[ G(t,x)=\partial_t\varphi_t(x)-\psi^+ +(A+1)\varphi_t(x).\]
 If $G$ achieves its minimum at $t=0$ then 
 \[G\geq \partial_t\varphi_{0}-\psi^+-(A+1)C\geq -\sup_X\psi^--(A+1)C\geq -C'. \]
Otherwise, $G$ achieves its minimum at $(t_0,x_0)$ with $t_0>0$. 
We observe that \[ \tr_{\omega_t}(\omega_X)\geq n\bigg( \frac{\omega_X^n}{\om_t^n}\bigg)^{\frac{1}{n}}=ne^{-\frac{1}{n}\partial_t\varphi_t}.\]
We compute, at $(t_0,x_0)$,
\[0\geq \left(\frac{\partial}{\partial t}-\Delta_t \right)G=(A+1)[\partial_t\varphi_t-n]+\tr_{\omega_t}(\omega_X)\geq(A+1)[\partial_t\varphi_t-n]+ne^{-\frac{1}{n}\partial_t\varphi_t}.  \] Thus $\partial_t\varphi_{t_0}(x_0)\geq -C(A,n)$. We infer that
\[G\geq -C(A,n)-\sup_X\psi^+-(A+1)C, \] which implies the desired estimate. 

The proof of (3) is very close to that of ~\cite[Lemma 5.5]{dang2026singularities}, which we include here for the readers' convenience. We set $\Tilde\varphi_t:=\varphi_t-\frac{1}{2A}\psi^{-}+C\geq 0$. We consider
\[Q(t,x)=\log\tr_{\omega_X}\omega_t(x)-M\Tilde{\varphi}_t(x)+e^{-\Tilde{\varphi}_t(x)}. \]
If $Q$  achieves its maximum at $t=0$, we are done since ${\omega_X}+\dc\varphi_0\leq Ce^{-(2A)^{-1}\psi^-}\omega_X$ by~\cite[Theorem 4.2]{guedj2021quasi}. Otherwise, $Q$ achieves its maximum at $(t_0,x_0)$ with $t_0>0$. In what follows, we compute at this point. For simplicity, we suppress the subscript $t$.
From~\cite[Proposition 3.1]{tosatti2015evolution} (also~\cite[(4.2)]{tosatti2015evolution}) we have
		\[ \left(\frac{\partial}{\partial t}-\Delta_{t}\right)\log\tr_{\omega_X}\omega\leq \frac{2}{(\tr_{\omega_X}\omega)^2}\textrm{Re}(g^{\Bar{q}k}(T_0)^p_{kp}\partial_{\Bar{q}}\tr_{\omega_X}\omega)+C\tr_\omega \omega_X,\]
		where $(T_0)^p_{kp}$ denote the torsion terms corresponding to $\omega_X$. At the maximum point $(t_0, x_0)$ of $Q$, we have $\partial_iQ=0$, hence \[\frac{1}{\tr_{\omega_X}\omega_t}\partial_i\tr_{\omega_X}\omega-M\partial_i\Tilde{\varphi}-e^{-\Tilde{\varphi}}\partial_i\Tilde{\varphi}=0.\]
		Thus, the Cauchy-Schwarz inequality yields
		\begin{equation*}
		\begin{split}
		\left| \frac{2}{(\tr_{\omega_X}\omega)^2}\textrm{Re}(g^{\Bar{q}k}(T_0)^p_{kp}\partial_{\Bar{q}}\tr_{\omega_X}\omega)\right|&\leq \left| \frac{2}{(\tr_{\omega_X}\omega)^2}\textrm{Re}((M+e^{-\Tilde{\varphi}})g^{\Bar{q}k}(T_0)^p_{kp}\partial_{\Bar{q}}\Tilde{\varphi}\right|\\
		&\leq e^{-\Tilde{\varphi}}|\partial\Tilde{\varphi}|^2_\omega+C(M+1)^2 e^{
        \Tilde{\varphi}}\frac{\tr_\omega \omega_X}{(\tr_{\omega_X}\omega)^2},
		\end{split}
		\end{equation*} for uniform $C>0$ only depending on the torsion term.
		Thus, it follows that, at the point $(t_0,x_0)$,
		\begin{equation}\label{eq: c2Q}
		\begin{split}
		0\leq \left(\frac{\partial}{\partial t}-\Delta_{t}\right)Q &\leq C(M+1)^2e^{\Tilde{\varphi}}\frac{\tr_\omega \omega_X}{(\tr_{\omega_X}\omega)^2}+C\tr_\omega \omega_X\\
		&\quad -(M+e^{-\Tilde{\varphi}})\partial_t{\f} +(M+e^{-\Tilde{\varphi}})\tr_\omega(\omega-(\omega_X+(2A)^{-1}\dc\psi^{-}))\\
		&\leq C(M+1)^2e^{\Tilde{\varphi}} \frac{\tr_\omega \omega_X}{(\tr_{\omega_X}\omega)^2}+(C-M/2)\tr_\omega \omega_X\\
        &\quad-(M+e^{-\Tilde{\varphi}})\partial_t\varphi+(M+1)n.
		\end{split}
		\end{equation}
         If at $(t_0,x_0)$, we have $(\tr_{\omega_X}   \omega)^2\leq e^{\Tilde{\varphi}} C(M+1)^2$ then at the same point \[Q\leq C+\frac{1}{2}\Tilde{\varphi}-M\Tilde{\varphi}+e^{-\Tilde{\varphi}}\leq C+1\] since $\Tilde{\varphi}\geq 0$, we are done. Otherwise, we choose $M=2(C+A)$. Hence, from~\eqref{eq: c2Q} one gets \begin{equation}\label{eq: Lap-final}
             \tr_{\omega}\omega_X\leq- (2C+2A+e^{-\Tilde\varphi})\partial_t\varphi_t+n(2C+2A+1). 
         \end{equation} 
         Recall that \[\tr_{\omega_X}\omega\leq n\left(\frac{\omega^n}{\omega_X^n} \right)(\tr_{\omega}\omega_X)^{n-1}=e^{\partial_t\varphi}(\tr_{\omega}\omega_X)^{n-1}. \]
We divide into two cases.
\begin{itemize}
    \item If $\partial_t\varphi_{t_0}(x_0)\geq 0$ then  \eqref{eq: Lap-final} becomes $tr_{\omega}\omega_X(t_0,x_0)\leq n(2C+2A+1)$. This implies that at $(t_0,x_0)$ \[{\rm tr}_{\omega_X}\omega\leq C'e^{C-\psi^-}, \]
    because by (2), $\partial_t\varphi\leq C-\psi^-$. It follows that
    \[Q\leq Q(t_0,x_0)\leq C'-\psi^- - 2(C+A)\varphi+(C/A+1)\psi^-+1\leq C''. \]
\item  If $\partial_t\varphi_{t_0}(x_0)\leq 0$, then at $(t_0,x_0)$
\[ \tr_{\omega_X}\omega\leq e^{\partial_t\varphi}(-C'\partial_t\varphi+C')^{n-1}\leq C''.\] Thus $Q$ is uniformly bounded from above.  
\end{itemize}
Since $\varphi_t$ is uniformly bounded, we obtain the desired estimate.

\medskip
\textbf{Case 2:} $\psi^{\pm}$ are merely $A\omega_X$-psh functions. Thanks to the regularization theorem~\cite{demailly1992regularization,blocki2007regularization}, we can find sequences $(\psi^{\pm}_j)_{j\geq 1}$ of smooth $A\omega_X$-psh functions decreasing pointwise to $\psi^{\pm}$. It follows from Tosatti--Weinkove's main theorem~\cite{tosatti2010complex} that there exist $\varphi_{0,j}\in\PSH(X,\omega_X)\cap\mathcal{C}^\infty(X)$ and $c_j\in\mathbb R$ which solve \[(\omega_X+\dc\varphi_{0,j})^n=e^{\psi^+_j-\psi^-_j+c_j}\omega_X^n,\quad\sup_X\varphi_{0,j}=0. \] We see that for $p>1$, $\|e^{-\psi^-_j}\|_p$ are uniformly bounded. Thus, it follows from~\cite{kolodziej2015weak} that $\|\varphi_{0,j}\|_{L^\infty}$ are uniformly bounded (independent of $j$) and $\varphi_{0,j}\to\varphi_0$ uniformly by the $L^\infty$-$L^1$ stability.
From~\cite{tosatti2015evolution,to2018regularizing}, there exists a smooth solution to the corresponding complex Monge--Amp\`ere flow
\[(\omega_X+\dc\varphi_{t,j})^n=e^{\partial_t\varphi_{t,j}}\omega_X^n,\quad \varphi_{t,j}|_{t=0}=\varphi_{0,j}. \]
We observe that $\varphi_{t,j}$ converges to $\varphi_t$ in $\mathcal{C}^\infty((0,T)\times X)$; see, for example, \cite[Theorem B]{to2018regularizing}. 
Since $\psi_j^{\pm}$ are decreasing, we obtain the desired estimates for $\varphi_{t,j}$, and hence also for $\varphi_t$.
           \end{proof}
        From our previous estimates, we apply the complex parabolic Evans--Krylov theory together with Schauder’s estimates (see also~\cite{sherman-weinkove13-estimates}) to obtain the following higher order a priori estimates.
\begin{proposition}
   Assume $\psi^{\pm}$ are smooth in some open Zariski set $\Omega\subset X$. For any compact subset $K\subset \Omega$ and $0<\varepsilon<S<T_{\max}$, there exist constants $C(k,S)$ such that \[\|\varphi\|_{\mathcal{C}^k([\varepsilon,S]\times K)}\leq C(k,S). \]
   Moreover, $\varphi_t$ converges to $\varphi_0$ in $\mathcal{C}^{\infty}_{\rm loc}(\Omega)$ as $t\to 0$.
\end{proposition}

    \begin{theorem}
        For $0<S<T_{\max}$ and any $x,y\in X$, we have
        \[ \textrm{diam}(X,\omega_t)\leq C,\qquad d_{\omega_t}(x,y)\leq Cd_{\omega_X}(x,y)^\alpha,\;\forall\, t\in (0,S],\]
        for some constants $C,\alpha>0$ that only depend on $X$, $\omega_X$, $A$, $S$, $p$ and an upper bound for $\|e^{-\psi^-}\|_p$.
    \end{theorem}   
    \begin{proof}
        By Proposition~\ref{prop: estimates} (2), $\|e^{\partial_t\varphi_t}\|_p\leq C\|e^{-\psi^-}\|_p$ is uniformly bounded in $t\in[0,S]$. The
result now follows from Theorem~\ref{thm: stability-L1} and Proposition~\ref{prop: diameter}. 
    \end{proof}

\bibliographystyle{alpha}
	\bibliography{bibfile}	

@article {bedford1976dirichlet,
    AUTHOR = {Bedford, E. and Taylor, B. A.},
     TITLE = {The {D}irichlet problem for a complex {M}onge-{A}mp\`ere
              equation},
   JOURNAL = {Invent. Math.},
  FJOURNAL = {Inventiones Mathematicae},
    VOLUME = {37},
      YEAR = {1976},
    NUMBER = {1},
     PAGES = {1--44},
      ISSN = {0020-9910},
   MRCLASS = {32F05},
MRREVIEWER = {Jacques Vauthier},
       DOI = {10.1007/BF01418826},
       URL = {https://doi.org/10.1007/BF01418826},
}

@article {bedford1982new,
    AUTHOR = {Bedford, E. and Taylor, B. A.},
     TITLE = {A new capacity for plurisubharmonic functions},
   JOURNAL = {Acta Math.},
  FJOURNAL = {Acta Mathematica},
    VOLUME = {149},
      YEAR = {1982},
    NUMBER = {1-2},
     PAGES = {1--40},
      ISSN = {0001-5962},
   MRCLASS = {32F05 (31C10 32C30)},
MRREVIEWER = {Guy Laville},
       DOI = {10.1007/BF02392348},
       URL = {https://doi.org/10.1007/BF02392348},
}

@article {Vu2018complex,
    AUTHOR = {Vu, D.-V.},
     TITLE = {Complex {M}onge-{A}mp\`ere equation for measures supported on
              real submanifolds},
   JOURNAL = {Math. Ann.},
  FJOURNAL = {Mathematische Annalen},
    VOLUME = {372},
      YEAR = {2018},
    NUMBER = {1-2},
     PAGES = {321--367},
      ISSN = {0025-5831,1432-1807},
   MRCLASS = {32W20 (32V20)},
  MRNUMBER = {3856815},
MRREVIEWER = {Tat\ Dat\ T\^{o}},
       DOI = {10.1007/s00208-017-1565-8},
       URL = {https://doi.org/10.1007/s00208-017-1565-8},
}

@article {blocki2007regularization,
    AUTHOR = {B{\l}ocki, Z. and Ko{\l}odziej, S.},
     TITLE = {On regularization of plurisubharmonic functions on manifolds},
   JOURNAL = {Proc. Amer. Math. Soc.},
  FJOURNAL = {Proceedings of the American Mathematical Society},
    VOLUME = {135},
      YEAR = {2007},
    NUMBER = {7},
     PAGES = {2089--2093},
      ISSN = {0002-9939},
   MRCLASS = {32U05 (32U25)},
MRREVIEWER = {Norman Levenberg},
       DOI = {10.1090/S0002-9939-07-08858-2},
       URL = {https://doi.org/10.1090/S0002-9939-07-08858-2},
}

@article {boucksom2010monge,
    AUTHOR = {Boucksom, S. and Eyssidieux, P. and Guedj,
              V. and Zeriahi, A.},
     TITLE = {Monge-{A}mp\`ere equations in big cohomology classes},
   JOURNAL = {Acta Math.},
  FJOURNAL = {Acta Mathematica},
    VOLUME = {205},
      YEAR = {2010},
    NUMBER = {2},
     PAGES = {199--262},
      ISSN = {0001-5962},
   MRCLASS = {32U40 (32Q20 32U15 32W20)},
  MRNUMBER = {2746347},
MRREVIEWER = {S\l awomir Dinew},
       DOI = {10.1007/s11511-010-0054-7},
       URL = {https://doi.org/10.1007/s11511-010-0054-7},
}

@article {cherrier1987equation,
    AUTHOR = {Cherrier, P.},
     TITLE = {\'{E}quations de {M}onge-{A}mp\`ere sur les vari\'{e}t\'{e}s hermitiennes
              compactes},
   JOURNAL = {Bull. Sci. Math. (2)},
  FJOURNAL = {Bulletin des Sciences Math\'{e}matiques. 2e S\'{e}rie},
    VOLUME = {111},
      YEAR = {1987},
    NUMBER = {4},
     PAGES = {343--385},
      ISSN = {0007-4497},
   MRCLASS = {58G30 (32C10 35J60 53C55)},
  MRNUMBER = {921559},
MRREVIEWER = {John M. Lee},
}

@article {hanani1996generalisation,
    AUTHOR = {Hanani, A.},
     TITLE = {Une g\'{e}n\'{e}ralisation de l'\'{e}quation de {M}onge-{A}mp\`ere sur les
              vari\'{e}t\'{e}s hermitiennes compactes},
   JOURNAL = {Bull. Sci. Math.},
  FJOURNAL = {Bulletin des Sciences Math\'{e}matiques},
    VOLUME = {120},
      YEAR = {1996},
    NUMBER = {2},
     PAGES = {215--252},
      ISSN = {0007-4497},
   MRCLASS = {32F07 (32J99 35J60)},
  MRNUMBER = {1387422},
MRREVIEWER = {Thierry Bouche},
}

@article{hanani1996equation,
    AUTHOR = {Hanani, A.},
     TITLE = {\'{E}quations du type de {M}onge-{A}mp\`ere sur les vari\'{e}t\'{e}s
              hermitiennes compactes},
   JOURNAL = {J. Funct. Anal.},
  FJOURNAL = {Journal of Functional Analysis},
    VOLUME = {137},
      YEAR = {1996},
    NUMBER = {1},
     PAGES = {49--75},
      ISSN = {0022-1236},
   MRCLASS = {32F07 (58G30)},
  MRNUMBER = {1383012},
MRREVIEWER = {S\l awomir Ko\l odziej},
       DOI = {10.1006/jfan.1996.0040},
       URL = {https://doi.org/10.1006/jfan.1996.0040},
}

@article {demailly1992regularization,
    AUTHOR = {Demailly, J.-P.},
     TITLE = {Regularization of closed positive currents and intersection
              theory},
   JOURNAL = {J. Algebraic Geom.},
  FJOURNAL = {Journal of Algebraic Geometry},
    VOLUME = {1},
      YEAR = {1992},
    NUMBER = {3},
     PAGES = {361--409},
      ISSN = {1056-3911},
   MRCLASS = {32C30 (32C17 32J25)},
MRREVIEWER = {Takeo Ohsawa},
}

@incollection {demailly1994regularization,
    AUTHOR = {Demailly, J.-P.},
     TITLE = {Regularization of closed positive currents of type {$(1,1)$}
              by the flow of a {C}hern connection},
 BOOKTITLE = {Contributions to complex analysis and analytic geometry},
    SERIES = {Aspects Math., E26},
     PAGES = {105--126},
 PUBLISHER = {Friedr. Vieweg, Braunschweig},
      YEAR = {1994},
   MRCLASS = {32C30 (32J25)},
  MRNUMBER = {1319346},
MRREVIEWER = {Mongi Blel},
}

@article {demailly2004numerical,
    AUTHOR = {Demailly, J.-P. and Paun, M.},
     TITLE = {Numerical characterization of the {K}\"{a}hler cone of a compact
              {K}\"{a}hler manifold},
   JOURNAL = {Ann. of Math. (2)},
  FJOURNAL = {Annals of Mathematics. Second Series},
    VOLUME = {159},
      YEAR = {2004},
    NUMBER = {3},
     PAGES = {1247--1274},
      ISSN = {0003-486X},
   MRCLASS = {32J27 (32Q15)},
  MRNUMBER = {2113021},
MRREVIEWER = {Philippe P. Eyssidieux},
       DOI = {10.4007/annals.2004.159.1247},
       URL = {https://doi.org/10.4007/annals.2004.159.1247},
}

@article {demailly2014holder,
    AUTHOR = {Demailly, J.-P. and Dinew, S. and Guedj, V.
              and Pham, H. H. and Ko{\l}odziej, S. and Zeriahi,
              A.},
     TITLE = {H\"{o}lder continuous solutions to {M}onge-{A}mp\`ere equations},
   JOURNAL = {J. Eur. Math. Soc. (JEMS)},
  FJOURNAL = {Journal of the European Mathematical Society (JEMS)},
    VOLUME = {16},
      YEAR = {2014},
    NUMBER = {4},
     PAGES = {619--647},
      ISSN = {1435-9855},
   MRCLASS = {32W20 (32Q15 32U05 32U15 32U40 35B65 35J96 53C55)},
  MRNUMBER = {3191972},
MRREVIEWER = {Muhammed Ali Alan},
       DOI = {10.4171/JEMS/442},
       URL = {https://doi.org/10.4171/JEMS/442},
}

@article{demaillycomplex,
  title={Complex Analytic and Differential Geometry},
  author={Demailly, J.-P.},
  journal={available at \href{https://www-fourier.ujf-grenoble.fr/~demailly/manuscripts/agbook.pdf}{https://www-fourier.ujf-grenoble.fr/~demailly/manuscripts/agbook.pdf}},
  year={2012},
URL={https://www-fourier.ujf-grenoble.fr/~Demailly/manuscripts/agbook.pdf}
}

@incollection {dinew2012pluripotential,
    AUTHOR = {Dinew, S. and Ko{\l}odziej, S.},
     TITLE = {Pluripotential estimates on compact {H}ermitian manifolds},
 BOOKTITLE = {Advances in geometric analysis},
    SERIES = {Adv. Lect. Math. (ALM)},
    VOLUME = {21},
     PAGES = {69--86},
 PUBLISHER = {Int. Press, Somerville, MA},
      YEAR = {2012},
   MRCLASS = {32U05 (32U40 32W20 53C55)},
MRREVIEWER = {Slimane Benelkourchi},
}

@article {di2017uniqueness,
    AUTHOR = {Di{ }Nezza, E. and Lu, C. H.},
     TITLE = {Uniqueness and short time regularity of the weak
              {K}\"{a}hler-{R}icci flow},
   JOURNAL = {Adv. Math.},
  FJOURNAL = {Advances in Mathematics},
    VOLUME = {305},
      YEAR = {2017},
     PAGES = {953--993},
      ISSN = {0001-8708},
   MRCLASS = {32Q15 (32U40 32W20 53C44 53C55)},
  MRNUMBER = {3570152},
MRREVIEWER = {Valentino Tosatti},
       DOI = {10.1016/j.aim.2016.10.011},
       URL = {https://doi.org/10.1016/j.aim.2016.10.011},
}

@article {eyssidieux2009singular,
    AUTHOR = {Eyssidieux, P. and Guedj, V. and Zeriahi, A.},
     TITLE = {Singular {K}\"{a}hler-{E}instein metrics},
   JOURNAL = {J. Amer. Math. Soc.},
  FJOURNAL = {Journal of the American Mathematical Society},
    VOLUME = {22},
      YEAR = {2009},
    NUMBER = {3},
     PAGES = {607--639},
      ISSN = {0894-0347},
   MRCLASS = {32Q20 (31C10 32J27 32Q25 32W20)},
  MRNUMBER = {2505296},
MRREVIEWER = {Zhou Zhang},
       DOI = {10.1090/S0894-0347-09-00629-8},
       URL = {https://doi.org/10.1090/S0894-0347-09-00629-8},
}

@article {fornaess1980levi,
    AUTHOR = {Forn{\ae}ss, J. E. and Narasimhan, R.},
     TITLE = {The {L}evi problem on complex spaces with singularities},
   JOURNAL = {Math. Ann.},
  FJOURNAL = {Mathematische Annalen},
    VOLUME = {248},
      YEAR = {1980},
    NUMBER = {1},
     PAGES = {47--72},
      ISSN = {0025-5831},
   MRCLASS = {32E10},
  MRNUMBER = {569410},
MRREVIEWER = {Peter Pflug},
       DOI = {10.1007/BF01349254},
       URL = {https://doi.org/10.1007/BF01349254},
}

@article {guedj2008holder,
    AUTHOR = {Guedj, V. and Kolodziej, S. and Zeriahi, A.},
     TITLE = {H\"{o}lder continuous solutions to {M}onge-{A}mp\'{e}re equations},
   JOURNAL = {Bull. Lond. Math. Soc.},
  FJOURNAL = {Bulletin of the London Mathematical Society},
    VOLUME = {40},
      YEAR = {2008},
    NUMBER = {6},
     PAGES = {1070--1080},
      ISSN = {0024-6093},
   MRCLASS = {32W20 (32U15 35J60)},
  MRNUMBER = {2471956},
MRREVIEWER = {Filippo Bracci},
       DOI = {10.1112/blms/bdn092},
       URL = {https://doi.org/10.1112/blms/bdn092},
}

@article {guedj2012stability,
    AUTHOR = {Guedj, V. and Zeriahi, A.},
     TITLE = {Stability of solutions to complex {M}onge-{A}mp\`ere equations
              in big cohomology classes},
   JOURNAL = {Math. Res. Lett.},
  FJOURNAL = {Mathematical Research Letters},
    VOLUME = {19},
      YEAR = {2012},
    NUMBER = {5},
     PAGES = {1025--1042},
      ISSN = {1073-2780},
   MRCLASS = {32W20 (32Q15 32U40)},
  MRNUMBER = {3039828},
MRREVIEWER = {S\l awomir Ko\l odziej},
       DOI = {10.4310/MRL.2012.v19.n5.a6},
       URL = {https://doi.org/10.4310/MRL.2012.v19.n5.a6},
}

@book {guedj2017degenerate,
    AUTHOR = {Guedj, V. and Zeriahi, A.},
     TITLE = {Degenerate complex {M}onge-{A}mp\`ere equations},
    SERIES = {EMS Tracts in Mathematics},
    VOLUME = {26},
 PUBLISHER = {European Mathematical Society (EMS), Z\"{u}rich},
      YEAR = {2017},
     PAGES = {xxiv+472},
      ISBN = {978-3-03719-167-5},
   MRCLASS = {32W20 (32Q20 32U15 32U20 32U40 35J96)},
MRREVIEWER = {Slimane Benelkourchi},
       DOI = {10.4171/167},
       URL = {https://doi.org/10.4171/167},
}

@article {guedj2017regularizing,
    AUTHOR = {Guedj, V. and Zeriahi, A.},
     TITLE = {Regularizing properties of the twisted {K}\"{a}hler-{R}icci flow},
   JOURNAL = {J. Reine Angew. Math.},
  FJOURNAL = {Journal f\"{u}r die Reine und Angewandte Mathematik. [Crelle's
              Journal]},
    VOLUME = {729},
      YEAR = {2017},
     PAGES = {275--304},
      ISSN = {0075-4102},
   MRCLASS = {53C44 (53C55)},
  MRNUMBER = {3680377},
MRREVIEWER = {Yuanqi Wang},
       DOI = {10.1515/crelle-2014-0105},
       URL = {https://doi.org/10.1515/crelle-2014-0105},
}

@article {guedj2020pluripotential,
    AUTHOR = {Guedj, V. and Lu, C. H. and Zeriahi, A.},
     TITLE = {Pluripotential {K}\"{a}hler-{R}icci flows},
   JOURNAL = {Geom. Topol.},
  FJOURNAL = {Geometry \& Topology},
    VOLUME = {24},
      YEAR = {2020},
    NUMBER = {3},
     PAGES = {1225--1296},
      ISSN = {1465-3060},
   MRCLASS = {53E30 (32W20 58J35)},
       DOI = {10.2140/gt.2020.24.1225},
       URL = {https://doi.org/10.2140/gt.2020.24.1225},
}

@article {kolodziej1998complex,
    AUTHOR = {Ko{\l}odziej, S.},
     TITLE = {The complex {M}onge-{A}mp\`ere equation},
   JOURNAL = {Acta Math.},
  FJOURNAL = {Acta Mathematica},
    VOLUME = {180},
      YEAR = {1998},
    NUMBER = {1},
     PAGES = {69--117},
      ISSN = {0001-5962},
   MRCLASS = {32F07 (32C17 35J60)},
  MRNUMBER = {1618325},
MRREVIEWER = {M. Klimek},
       DOI = {10.1007/BF02392879},
       URL = {https://doi.org/10.1007/BF02392879},
}

@incollection {kolodziej2015weak,
    AUTHOR = {Ko{\l}odziej, S. and Nguyen, N. C.},
     TITLE = {Weak solutions to the complex {M}onge-{A}mp\`ere equation on
              {H}ermitian manifolds},
 BOOKTITLE = {Analysis, complex geometry, and mathematical physics: in honor
              of {D}uong {H}. {P}hong},
    SERIES = {Contemp. Math.},
    VOLUME = {644},
     PAGES = {141--158},
 PUBLISHER = {Amer. Math. Soc., Providence, RI},
      YEAR = {2015},
   MRCLASS = {32W20 (32U40)},
MRREVIEWER = {Ly Kim Ha},
       DOI = {10.1090/conm/644/12775},
       URL = {https://doi.org/10.1090/conm/644/12775},
}

@article {kolodziej2018holder,
    AUTHOR = {Ko{\l}odziej, S. and Nguyen, N. C.},
     TITLE = {H\"{o}lder continuous solutions of the {M}onge-{A}mp\`ere equation
              on compact {H}ermitian manifolds},
   JOURNAL = {Ann. Inst. Fourier (Grenoble)},
  FJOURNAL = {Universit\'{e} de Grenoble. Annales de l'Institut Fourier},
    VOLUME = {68},
      YEAR = {2018},
    NUMBER = {7},
     PAGES = {2951--2964},
      ISSN = {0373-0956},
   MRCLASS = {53C55 (32U40 35J96 35R01)},
MRREVIEWER = {Nguyen Xuan Hong},
       URL = {http://aif.cedram.org/item?id=AIF_2018__68_7_2951_0},
}

@article {kolodziej2019stability,
    AUTHOR = {Ko{\l}odziej, S. and Nguyen, N. C.},
     TITLE = {Stability and regularity of solutions of the {M}onge-{A}mp\`ere
              equation on {H}ermitian manifolds},
   JOURNAL = {Adv. Math.},
  FJOURNAL = {Advances in Mathematics},
    VOLUME = {346},
      YEAR = {2019},
     PAGES = {264--304},
      ISSN = {0001-8708},
   MRCLASS = {53C55 (32U40 32W20 35J96)},
MRREVIEWER = {Masaya Kawamura},
       DOI = {10.1016/j.aim.2019.02.004},
       URL = {https://doi.org/10.1016/j.aim.2019.02.004},
}

@article {kolodziej2021continuous,
    AUTHOR = {Ko{\l}odziej, S. and Nguyen, N. C.},
     TITLE = {Continuous solutions to {M}onge-{A}mp\`ere equations on
              {H}ermitian manifolds for measures dominated by capacity},
   JOURNAL = {Calc. Var. Partial Differential Equations},
  FJOURNAL = {Calculus of Variations and Partial Differential Equations},
    VOLUME = {60},
      YEAR = {2021},
    NUMBER = {3},
     PAGES = {Paper No. 93, 18},
      ISSN = {0944-2669},
   MRCLASS = {32U40 (32W20 53C55)},
  MRNUMBER = {4249871},
       DOI = {10.1007/s00526-021-01944-4},
       URL = {https://doi.org/10.1007/s00526-021-01944-4},
}

@article {nguyen2016complex,
    AUTHOR = {Nguyen, N. C.},
     TITLE = {The complex {M}onge-{A}mp\`ere type equation on compact
              {H}ermitian manifolds and applications},
   JOURNAL = {Adv. Math.},
  FJOURNAL = {Advances in Mathematics},
    VOLUME = {286},
      YEAR = {2016},
     PAGES = {240--285},
      ISSN = {0001-8708},
   MRCLASS = {32W20 (32U05 32U40 53C55)},
MRREVIEWER = {Bianca Santoro},
       DOI = {10.1016/j.aim.2015.09.009},
       URL = {https://doi.org/10.1016/j.aim.2015.09.009},
}

@article {to2018regularizing,
    AUTHOR = {Tô, T. D.},
     TITLE = {Regularizing properties of complex {M}onge-{A}mp\`ere flows
              {II}: {H}ermitian manifolds},
   JOURNAL = {Math. Ann.},
  FJOURNAL = {Mathematische Annalen},
    VOLUME = {372},
      YEAR = {2018},
    NUMBER = {1-2},
     PAGES = {699--741},
      ISSN = {0025-5831},
   MRCLASS = {32W20 (32U25 53C44)},
MRREVIEWER = {Ngoc Cuong Nguyen},
       DOI = {10.1007/s00208-017-1574-7},
       URL = {https://doi.org/10.1007/s00208-017-1574-7},
}

@article {tosatti2010complex,
    AUTHOR = {Tosatti, V. and Weinkove, B.},
     TITLE = {The complex {M}onge-{A}mp\`ere equation on compact {H}ermitian
              manifolds},
   JOURNAL = {J. Amer. Math. Soc.},
  FJOURNAL = {Journal of the American Mathematical Society},
    VOLUME = {23},
      YEAR = {2010},
    NUMBER = {4},
     PAGES = {1187--1195},
      ISSN = {0894-0347},
   MRCLASS = {32W20 (32Q15)},
  MRNUMBER = {2669712},
MRREVIEWER = {S\l awomir Ko\l odziej},
       DOI = {10.1090/S0894-0347-2010-00673-X},
       URL = {https://doi.org/10.1090/S0894-0347-2010-00673-X},
}

@article {tosatti2010estimates,
    AUTHOR = {Tosatti, V. and Weinkove, B.},
     TITLE = {Estimates for the complex {M}onge-{A}mp\`ere equation on
              {H}ermitian and balanced manifolds},
   JOURNAL = {Asian J. Math.},
  FJOURNAL = {Asian Journal of Mathematics},
    VOLUME = {14},
      YEAR = {2010},
    NUMBER = {1},
     PAGES = {19--40},
      ISSN = {1093-6106},
   MRCLASS = {32W20 (32Q25 35J96 35R01)},
  MRNUMBER = {2726593},
MRREVIEWER = {Yanir A. Rubinstein},
       DOI = {10.4310/AJM.2010.v14.n1.a3},
       URL = {https://doi.org/10.4310/AJM.2010.v14.n1.a3},
}

@article {tosatti2015evolution,
    AUTHOR = {Tosatti, V. and Weinkove, B.},
     TITLE = {On the evolution of a {H}ermitian metric by its
              {C}hern-{R}icci form},
   JOURNAL = {J. Differential Geom.},
  FJOURNAL = {Journal of Differential Geometry},
    VOLUME = {99},
      YEAR = {2015},
    NUMBER = {1},
     PAGES = {125--163},
      ISSN = {0022-040X},
   MRCLASS = {53C44 (53C55)},
MRREVIEWER = {Chengjie Yu},
       URL = {http://projecteuclid.org/euclid.jdg/1418345539},
}

@article {yau1978ricci,
    AUTHOR = {Yau, S. T.},
     TITLE = {On the {R}icci curvature of a compact {K}\"{a}hler manifold and
              the complex {M}onge-{A}mp\`ere equation. {I}},
   JOURNAL = {Comm. Pure Appl. Math.},
  FJOURNAL = {Communications on Pure and Applied Mathematics},
    VOLUME = {31},
      YEAR = {1978},
    NUMBER = {3},
     PAGES = {339--411},
      ISSN = {0010-3640},
   MRCLASS = {53C55 (32C10 35J60)},
  MRNUMBER = {480350},
MRREVIEWER = {Robert E. Greene},
       DOI = {10.1002/cpa.3160310304},
       URL = {https://doi.org/10.1002/cpa.3160310304},
}

@article{guedj2021quasi,
   AUTHOR = {Guedj, V. and Lu, C. H.},
     TITLE = {Quasi-plurisubharmonic envelopes 3: {S}olving
              {M}onge-{A}mp\`ere equations on hermitian manifolds},
   JOURNAL = {J. Reine Angew. Math.},
  FJOURNAL = {Journal f\"ur die Reine und Angewandte Mathematik. [Crelle's
              Journal]},
    VOLUME = {800},
      YEAR = {2023},
     PAGES = {259--298},
      ISSN = {0075-4102,1435-5345},
   MRCLASS = {32W20},
  MRNUMBER = {4609828},
MRREVIEWER = {S\l awomir\ Dinew},
       DOI = {10.1515/crelle-2023-0030},
       URL = {https://doi.org/10.1515/crelle-2023-0030},
}

@article{dang2026singularities,
    AUTHOR = {Dang, Q.-T.},
     TITLE = {{Singularities of the Chern-Ricci flow}},
   JOURNAL = {Anal. PDE},
  FJOURNAL = {Analysis \& PDE},
    VOLUME = {19},
      YEAR = {2026},
    NUMBER = {3},
     PAGES = {449--483},
}

@article {sherman-weinkove13-estimates,
    AUTHOR = {Sherman, M. and Weinkove, B.},
     TITLE = {Local {C}alabi and curvature estimates for the {C}hern-{R}icci
              flow},
   JOURNAL = {New York J. Math.},
  FJOURNAL = {New York Journal of Mathematics},
    VOLUME = {19},
      YEAR = {2013},
     PAGES = {565--582},
      ISSN = {1076-9803},
   MRCLASS = {53C44 (53C55)},
  MRNUMBER = {3119098},
MRREVIEWER = {Yanir\ A.\ Rubinstein},
       URL = {http://nyjm.albany.edu:8000/j/2013/19_565.html},
}

@article {fino2009blow,
    AUTHOR = {Fino, A. and Tomassini, A.},
     TITLE = {Blow-ups and resolutions of strong {K}\"{a}hler with torsion
              metrics},
   JOURNAL = {Adv. Math.},
  FJOURNAL = {Advances in Mathematics},
    VOLUME = {221},
      YEAR = {2009},
    NUMBER = {3},
     PAGES = {914--935},
      ISSN = {0001-8708},
   MRCLASS = {32J27 (53C15 53C55)},
  MRNUMBER = {2511043},
MRREVIEWER = {Valentino Tosatti},
       DOI = {10.1016/j.aim.2009.02.001},
       URL = {https://doi.org/10.1016/j.aim.2009.02.001},
}

@article{guedj2021quasi1,
   AUTHOR = {Guedj, V. and Lu, C. H.},
     TITLE = {Quasi-plurisubharmonic envelopes 1: uniform estimates on
              {K}\"ahler manifolds},
   JOURNAL = {J. Eur. Math. Soc. (JEMS)},
  FJOURNAL = {Journal of the European Mathematical Society (JEMS)},
    VOLUME = {27},
      YEAR = {2025},
    NUMBER = {3},
     PAGES = {1185--1208},
      ISSN = {1435-9855,1435-9863},
   MRCLASS = {32Q15 (32U05 32W20 35A23)},
  MRNUMBER = {4874940},
MRREVIEWER = {Ngoc\ Cuong\ Nguyen},
       DOI = {10.4171/jems/1460},
       URL = {https://doi.org/10.4171/jems/1460},
}

@article {guedj2022quasi,
    AUTHOR = {Guedj, V. and Lu, C. H.},
     TITLE = {Quasi-plurisubharmonic envelopes 2: {B}ounds on
              {M}onge-{A}mp\`ere volumes},
   JOURNAL = {Algebr. Geom.},
  FJOURNAL = {Algebraic Geometry},
    VOLUME = {9},
      YEAR = {2022},
    NUMBER = {6},
     PAGES = {688--713},
      ISSN = {2313-1691},
   MRCLASS = {32W20 (32J18 32U05 35A23 35J96)},
  MRNUMBER = {4518244},
MRREVIEWER = {S\l awomir Dinew},
       DOI = {10.14231/ag-2022-021},
       URL = {https://doi.org/10.14231/ag-2022-021},
}

@article {lu2021stability,
    AUTHOR = {Lu, C. H. and Phung, T.-T. and T\^{o}, T.-D.},
     TITLE = {Stability and {H}\"{o}lder regularity of solutions to complex
              {M}onge-{A}mp\`ere equations on compact {H}ermitian manifolds},
   JOURNAL = {Ann. Inst. Fourier (Grenoble)},
  FJOURNAL = {Universit\'{e} de Grenoble. Annales de l'Institut Fourier},
    VOLUME = {71},
      YEAR = {2021},
    NUMBER = {5},
     PAGES = {2019--2045},
      ISSN = {0373-0956},
   MRCLASS = {32W20 (32Q15 32U05)},
  MRNUMBER = {4398254},
MRREVIEWER = {Rafa\l  Czy\.{z}},
       DOI = {10.5802/aif.3436},
       URL = {https://doi.org/10.5802/aif.3436},
}

@article {ChengXu2025-viscosity-hessian,
    AUTHOR = {Cheng, J. and Xu, Y.},
     TITLE = {Viscosity solution to complex {H}essian equations on compact
              {H}ermitian manifolds},
   JOURNAL = {J. Funct. Anal.},
  FJOURNAL = {Journal of Functional Analysis},
    VOLUME = {289},
      YEAR = {2025},
    NUMBER = {5},
     PAGES = {Paper No. 110936, 52},
      ISSN = {0022-1236,1096-0783},
   MRCLASS = {32W50 (53C55 58J05)},
  MRNUMBER = {4885961},
MRREVIEWER = {S\l awomir\ Dinew},
       DOI = {10.1016/j.jfa.2025.110936},
       URL = {https://doi.org/10.1016/j.jfa.2025.110936},
}

@article{GuoPhongTongWang2021-modulus,
  title={{On the modulus of continuity of solutions to complex
 Monge-Amp\`ere equations}},
  author={Guo, B. and Phong, D. H. and Tong, F. and Wang, C.},
  journal={arXiv:2112.02354},
  year={2021}
}

@article{Fang25-continuity-hessian,
  title={{Continuity of solutions to complex Hessian equations on compact Hermitian manifolds}},
  author={Fang, Y.},
  journal={\href{https://arxiv.org/abs/2510.14690}{arXiv:2510.14690}, accepted in IUMJ},
  year={2025}
}

@article {GuedjGuenanciaZeriahi25-diameter,
    AUTHOR = {Guedj, V. and Guenancia, H. and Zeriahi, A.},
     TITLE = {Diameter of {K}\"ahler currents},
   JOURNAL = {J. Reine Angew. Math.},
  FJOURNAL = {Journal f\"ur die Reine und Angewandte Mathematik. [Crelle's
              Journal]},
    VOLUME = {820},
      YEAR = {2025},
     PAGES = {115--152},
      ISSN = {0075-4102,1435-5345},
   MRCLASS = {53C55 (32Q15 32U40 32W20 58A25)},
  MRNUMBER = {4871394},
       DOI = {10.1515/crelle-2024-0092},
       URL = {https://doi.org/10.1515/crelle-2024-0092},
}

@article{BoucksomGuedjLu2025-volume,
  title={{Volumes of Bott–Chern Classes}},
  author={Boucksom, S. and Guedj, V. and Lu, C.H. },
  journal={Peking Math J},
  year={2025}
}

@article{Zeriahi20-continuity,
  title={{Remarks on the modulus of continuity of subharmonic functions}},
  author={Zeriahi, A.},
  journal={\href{https://arxiv.org/abs/2007.08399}{arXiv:2007.08399}},
  year={2020}
}

@article {ChoChoi25-continuity,
    AUTHOR = {Cho, Y.-W. L. and Choi, Y.-J.},
     TITLE = {Continuity of solutions to complex {M}onge--{A}mp\`ere
              equations on compact {K}\"ahler spaces},
   JOURNAL = {Math. Ann.},
  FJOURNAL = {Mathematische Annalen},
    VOLUME = {393},
      YEAR = {2025},
    NUMBER = {1},
     PAGES = {807--830},
      ISSN = {0025-5831,1432-1807},
   MRCLASS = {32W20 (14J17 32Q20 32U20)},
  MRNUMBER = {4966573},
       DOI = {10.1007/s00208-025-03268-6},
       URL = {https://doi.org/10.1007/s00208-025-03268-6},
}

@article {Dinew-Zhang10-stability,
    AUTHOR = {Dinew, S. and Zhang, Z.},
     TITLE = {On stability and continuity of bounded solutions of degenerate
              complex {M}onge-{A}mp\`ere equations over compact {K}\"ahler
              manifolds},
   JOURNAL = {Adv. Math.},
  FJOURNAL = {Advances in Mathematics},
    VOLUME = {225},
      YEAR = {2010},
    NUMBER = {1},
     PAGES = {367--388},
      ISSN = {0001-8708,1090-2082},
   MRCLASS = {32W20 (32Q15)},
  MRNUMBER = {2669357},
MRREVIEWER = {Bianca\ Santoro},
       DOI = {10.1016/j.aim.2010.03.001},
       URL = {https://doi.org/10.1016/j.aim.2010.03.001},
}

@article {WWZ20-estimate,
    AUTHOR = {Wang, J. and Wang, X.-J. and Zhou, B.},
     TITLE = {A priori estimate for the complex {M}onge-{A}mp\`ere equation},
   JOURNAL = {Peking Math. J.},
  FJOURNAL = {Peking Mathematical Journal},
    VOLUME = {4},
      YEAR = {2021},
    NUMBER = {1},
     PAGES = {143--157},
      ISSN = {2096-6075,2524-7182},
   MRCLASS = {32W20 (35J96)},
  MRNUMBER = {4249053},
MRREVIEWER = {Hichame\ Amal},
       DOI = {10.1007/s42543-020-00025-3},
       URL = {https://doi.org/10.1007/s42543-020-00025-3},
}

@article {Kovats99-elliptic,
    AUTHOR = {Kovats, J.},
     TITLE = {{Dini-Campanato spaces and applications to nonlinear elliptic equation}},
   JOURNAL = {Electron. J. Differential Equations},
  FJOURNAL = {Electronic Journal Differential Equations},
      YEAR = {1999},
     PAGES = {p. No. 37, 20pp},
}

@article {Li_yang_2021,
    AUTHOR = {Li, Y.},
     TITLE = {On collapsing {C}alabi-{Y}au fibrations},
   JOURNAL = {J. Differential Geom.},
  FJOURNAL = {Journal of Differential Geometry},
    VOLUME = {117},
      YEAR = {2021},
    NUMBER = {3},
     PAGES = {451--483},
      ISSN = {0022-040X},
   MRCLASS = {53C55 (32Q25 53C25)},
  MRNUMBER = {4255068},
       DOI = {10.4310/jdg/1615487004},
       URL = {https://doi.org/10.4310/jdg/1615487004},
}

@article {dang24-chern,
    AUTHOR = {Dang, Q.-T.},
     TITLE = {Pluripotential {C}hern-{R}icci flows},
   JOURNAL = {Indiana Univ. Math. J.},
  FJOURNAL = {Indiana University Mathematics Journal},
    VOLUME = {73},
      YEAR = {2024},
    NUMBER = {4},
     PAGES = {1401--1441},
      ISSN = {0022-2518,1943-5258},
   MRCLASS = {32J18 (32Uxx 32W20 53E30)},
  MRNUMBER = {4806754},
}

@article{deruelle2025k,
  title={{K\"ahler-Ricci flows coming out of metric spaces}},
  author={Deruelle, A. and Guedj, V. and Guenancia, H. and Zeriahi, A.},
  journal={\href{https://arxiv.org/abs/2511.13473v1}{arXiv:2511.13473}},
  year={2025}
}

@article {Cheng-Xu24-m-subharmonic,
    AUTHOR = {Cheng, J. and Xu, Y.},
     TITLE = {Regularization of {$m$}-subharmonic functions and {H}\"older
              continuity},
   JOURNAL = {Math. Res. Lett.},
  FJOURNAL = {Mathematical Research Letters},
    VOLUME = {31},
      YEAR = {2024},
    NUMBER = {4},
     PAGES = {951--984},
      ISSN = {1073-2780,1945-001X},
   MRCLASS = {32U05 (32W50)},
  MRNUMBER = {4831045},
MRREVIEWER = {Ngoc\ Cuong\ Nguyen},
       DOI = {10.4310/mrl.241118233017},
       URL = {https://doi.org/10.4310/mrl.241118233017},
}

@article {GGZ23-continuity,
    AUTHOR = {Guedj, V. and Guenancia, H. and Zeriahi, A.},
     TITLE = {Continuity of singular {K}\"ahler-{E}instein potentials},
   JOURNAL = {Int. Math. Res. Not. IMRN},
  FJOURNAL = {International Mathematics Research Notices. IMRN},
      YEAR = {2023},
    NUMBER = {2},
     PAGES = {1355--1377},
      ISSN = {1073-7928,1687-0247},
   MRCLASS = {32Q20 (32W20)},
  MRNUMBER = {4537327},
MRREVIEWER = {Soufian\ Abja},
       DOI = {10.1093/imrn/rnab294},
       URL = {https://doi.org/10.1093/imrn/rnab294},
}

@article {WZ24-trace,
    AUTHOR = {Wang, J. and Zhou, B.},
     TITLE = {Trace inequalities, isocapacitary inequalities, and regularity
              of the complex {H}essian equations},
   JOURNAL = {Sci. China Math.},
  FJOURNAL = {Science China. Mathematics},
    VOLUME = {67},
      YEAR = {2024},
    NUMBER = {3},
     PAGES = {557--576},
      ISSN = {1674-7283,1869-1862},
   MRCLASS = {32W20 (32U05)},
  MRNUMBER = {4731277},
MRREVIEWER = {Ngoc\ Cuong\ Nguyen},
       DOI = {10.1007/s11425-022-2100-1},
       URL = {https://doi.org/10.1007/s11425-022-2100-1},
}

@article {guo2022-local,
    AUTHOR = {Guo, B. and Song, J.},
     TITLE = {Local noncollapsing for complex {M}onge-{A}mp\`ere equations},
   JOURNAL = {J. Reine Angew. Math.},
  FJOURNAL = {Journal f\"ur die Reine und Angewandte Mathematik. [Crelle's
              Journal]},
    VOLUME = {793},
      YEAR = {2022},
     PAGES = {225--238},
      ISSN = {0075-4102,1435-5345},
   MRCLASS = {53E20 (53C55)},
  MRNUMBER = {4513167},
MRREVIEWER = {Rare\c s\ R\u asdeaconu},
       DOI = {10.1515/crelle-2022-0069},
       URL = {https://doi.org/10.1515/crelle-2022-0069},
}

@article {Guo2024-diameter2,
    AUTHOR = {Guo, B. and Phong, D. H. and Song, J. and Sturm, J.},
     TITLE = {Diameter estimates in {K}\"ahler geometry {II}: removing the
              small degeneracy assumption},
   JOURNAL = {Math. Z.},
  FJOURNAL = {Mathematische Zeitschrift},
    VOLUME = {308},
      YEAR = {2024},
    NUMBER = {3},
     PAGES = {Paper No. 43, 7},
      ISSN = {0025-5874,1432-1823},
   MRCLASS = {53C55 (53C20 53C21)},
  MRNUMBER = {4805061},
       DOI = {10.1007/s00209-024-03600-x},
       URL = {https://doi.org/10.1007/s00209-024-03600-x},
}

@article {Fu-Guo-Song2020-geometric,
    AUTHOR = {Fu, X. and Guo, B. and Song, J.},
     TITLE = {Geometric estimates for complex {M}onge-{A}mp\`ere equations},
   JOURNAL = {J. Reine Angew. Math.},
  FJOURNAL = {Journal f\"ur die Reine und Angewandte Mathematik. [Crelle's
              Journal]},
    VOLUME = {765},
      YEAR = {2020},
     PAGES = {69--99},
      ISSN = {0075-4102,1435-5345},
   MRCLASS = {53C55 (32W20 35J96 53C25)},
  MRNUMBER = {4129356},
MRREVIEWER = {Sibel\ \c Sahin},
       DOI = {10.1515/crelle-2019-0020},
       URL = {https://doi.org/10.1515/crelle-2019-0020},
}

@article {guo2022-diameter,
    AUTHOR = {Guo, B. and Phong, D. H. and Song, J. and Sturm, J.},
     TITLE = {Diameter estimates in {K}\"ahler geometry},
   JOURNAL = {Comm. Pure Appl. Math.},
  FJOURNAL = {Communications on Pure and Applied Mathematics},
    VOLUME = {77},
      YEAR = {2024},
    NUMBER = {8},
     PAGES = {3520--3556},
      ISSN = {0010-3640,1097-0312},
   MRCLASS = {53C55 (53C20 53C21)},
  MRNUMBER = {4764747},
       DOI = {10.1002/cpa.22196},
       URL = {https://doi.org/10.1002/cpa.22196},
}

@article{do2023log,
 author={Do, H.-S and Vu, D.-V.},
  title={{Log continuity of solutions of complex Monge-Ampere equations}},
  journal={\href{https://arxiv.org/abs/2312.04128}{arXiv:2312.04128}},
  year={2023}
}

@article{vu24-diameter,
 author={Vu, D.-V.},
  title={{ Uniform Diameter and Non-collapsing Estimates for K\"ahler metrics}},
  JOURNAL = {J. Geom. Anal.},
  FJOURNAL = {Journal of Geometric Analysis},
    VOLUME = {36},
      YEAR = {2026},
    NUMBER = {75},
}

@article{nguyen-vu24-diameter,
 author={Nguyen, D.-B. and Vu, D.-V.},
  title={{ Uniform diameter estimates for K\"ahler metrics in big cohomology classes}},
  journal={\href{https://arxiv.org/abs/2410.18532}{arXiv:2410.1853}},
  year={2024}}

@article {GJSS25-cscK,
    AUTHOR = {Guo, B. and Jian, W. and Shi, Y. and Song, J.},
     TITLE = {Csc{K} metrics near the canonical class},
   JOURNAL = {J. Reine Angew. Math.},
  FJOURNAL = {Journal f\"ur die Reine und Angewandte Mathematik. [Crelle's
              Journal]},
    VOLUME = {820},
      YEAR = {2025},
     PAGES = {153--165},
      ISSN = {0075-4102,1435-5345},
   MRCLASS = {32Q15},
  MRNUMBER = {4871395},
MRREVIEWER = {Rare\c s\ R\u asdeaconu},
       DOI = {10.1515/crelle-2024-0096},
       URL = {https://doi.org/10.1515/crelle-2024-0096},
}

@article {dinh2014characterization,
    AUTHOR = {Dinh, T.-C. and Nguy\^{e}n, V.-A.},
     TITLE = {Characterization of {M}onge-{A}mp\`ere measures with {H}\"{o}lder
              continuous potentials},
   JOURNAL = {J. Funct. Anal.},
  FJOURNAL = {Journal of Functional Analysis},
    VOLUME = {266},
      YEAR = {2014},
    NUMBER = {1},
     PAGES = {67--84},
      ISSN = {0022-1236},
   MRCLASS = {32W20 (32Q15 35B65 35J96)},
  MRNUMBER = {3121721},
MRREVIEWER = {Rafa\l  Czy\.{z}},
       DOI = {10.1016/j.jfa.2013.08.026},
       URL = {https://doi.org/10.1016/j.jfa.2013.08.026},
}

@article {Guo-Phong-Sturm24-green,
    AUTHOR = {Guo, B. and Phong, D. H. and Sturm, J.},
     TITLE = {Green's functions and complex {M}onge-{A}mp\`ere equations},
   JOURNAL = {J. Differential Geom.},
  FJOURNAL = {Journal of Differential Geometry},
    VOLUME = {127},
      YEAR = {2024},
    NUMBER = {3},
     PAGES = {1083--1119},
      ISSN = {0022-040X,1945-743X},
   MRCLASS = {53C55 (53C21)},
  MRNUMBER = {4773174},
MRREVIEWER = {Sibel\ \c Sahin},
       DOI = {10.4310/jdg/1721071497},
       URL = {https://doi.org/10.4310/jdg/1721071497},
}

@article {KN16-hessian,
    AUTHOR = {Ko{\l}odziej, S. and Nguyen, N. C.},
     TITLE = {Weak solutions of complex {H}essian equations on compact
              {H}ermitian manifolds},
   JOURNAL = {Compos. Math.},
  FJOURNAL = {Compositio Mathematica},
    VOLUME = {152},
      YEAR = {2016},
    NUMBER = {11},
     PAGES = {2221--2248},
      ISSN = {0010-437X,1570-5846},
   MRCLASS = {32U40 (35D30 35J96 53C55)},
  MRNUMBER = {3577893},
MRREVIEWER = {Dongrui\ Wan},
       DOI = {10.1112/S0010437X16007417},
       URL = {https://doi.org/10.1112/S0010437X16007417},
}

@article {BKPZ16-holder,
    AUTHOR = {Baracco, L. and Khanh, T. V. and Pinton, S. and
              Zampieri, G.},
     TITLE = {H\"older regularity of the solution to the complex
              {M}onge-{A}mp\`ere equation with {$L^p$} density},
   JOURNAL = {Calc. Var. Partial Differential Equations},
  FJOURNAL = {Calculus of Variations and Partial Differential Equations},
    VOLUME = {55},
      YEAR = {2016},
    NUMBER = {4},
     PAGES = {Art. 74, 8},
      ISSN = {0944-2669,1432-0835},
   MRCLASS = {32U05 (32U40 53C55)},
  MRNUMBER = {3513212},
MRREVIEWER = {Ngoc\ Cuong\ Nguyen},
       DOI = {10.1007/s00526-016-1008-5},
       URL = {https://doi.org/10.1007/s00526-016-1008-5},
}

@article {FengShiXu18-dirichlet,
    AUTHOR = {Feng, K. and Shi, Y. L. and Xu, Y. Y.},
     TITLE = {On the {D}irichlet problem for a class of singular complex
              {M}onge-{A}mp\`ere equations},
   JOURNAL = {Acta Math. Sin. (Engl. Ser.)},
  FJOURNAL = {Acta Mathematica Sinica (English Series)},
    VOLUME = {34},
      YEAR = {2018},
    NUMBER = {2},
     PAGES = {209--220},
      ISSN = {1439-8516,1439-7617},
   MRCLASS = {35J96 (35J75 53C55)},
  MRNUMBER = {3750393},
       DOI = {10.1007/s10114-017-7148-5},
       URL = {https://doi.org/10.1007/s10114-017-7148-5},
}

@article {HongThuy18-holder,
    AUTHOR = {Hong, N. X. and Thuy, T. V.},
     TITLE = {H\"older continuous solutions to the complex
              {M}onge-{A}mp\`ere equations in non-smooth pseudoconvex
              domains},
   JOURNAL = {Anal. Math. Phys.},
  FJOURNAL = {Analysis and Mathematical Physics},
    VOLUME = {8},
      YEAR = {2018},
    NUMBER = {3},
     PAGES = {465--484},
      ISSN = {1664-2368,1664-235X},
   MRCLASS = {32U05 (32W20 35J96)},
  MRNUMBER = {3842208},
MRREVIEWER = {Szymon\ Pli\'s},
       DOI = {10.1007/s13324-017-0175-7},
       URL = {https://doi.org/10.1007/s13324-017-0175-7},
}

@article {Nguyen20-Holder,
    AUTHOR = {Nguyen, N. C.},
     TITLE = {On the {H}\"older continuous subsolution problem for the
              complex {M}onge-{A}mp\`ere equation, {II}},
   JOURNAL = {Anal. PDE},
  FJOURNAL = {Analysis \& PDE},
    VOLUME = {13},
      YEAR = {2020},
    NUMBER = {2},
     PAGES = {435--453},
      ISSN = {2157-5045,1948-206X},
   MRCLASS = {32U40 (35J96 53C55)},
  MRNUMBER = {4078232},
MRREVIEWER = {Sibel\ \c Sahin},
       DOI = {10.2140/apde.2020.13.435},
       URL = {https://doi.org/10.2140/apde.2020.13.435},
}

@article {MR4125759,
    AUTHOR = {Benali, A. and Zeriahi, A.},
     TITLE = {The {H}\"older continuous subsolution theorem for complex
              {H}essian equations},
   JOURNAL = {J. \'Ec. polytech. Math.},
  FJOURNAL = {Journal de l'\'Ecole polytechnique. Math\'ematiques},
    VOLUME = {7},
      YEAR = {2020},
     PAGES = {981--1007},
      ISSN = {2429-7100,2270-518X},
   MRCLASS = {32W20 (31C45 32U15 32U40 35B65 35J96)},
  MRNUMBER = {4125759},
MRREVIEWER = {Ngoc\ Cuong\ Nguyen},
       DOI = {10.5802/jep.133},
       URL = {https://doi.org/10.5802/jep.133},
}

@article {GuNguyen16-hessian,
    AUTHOR = {Gu, D. and Nguyen, N. C.},
     TITLE = {The {D}irichlet problem for a complex {H}essian equation on
              compact {H}ermitian manifolds with boundary},
   JOURNAL = {Ann. Sc. Norm. Super. Pisa Cl. Sci. (5)},
  FJOURNAL = {Annali della Scuola Normale Superiore di Pisa. Classe di
              Scienze. Serie V},
    VOLUME = {18},
      YEAR = {2018},
    NUMBER = {4},
     PAGES = {1189--1248},
      ISSN = {0391-173X,2036-2145},
   MRCLASS = {32W50 (32Q15 32U40 35J96 53C55)},
  MRNUMBER = {3829745},
MRREVIEWER = {Yanir\ A.\ Rubinstein},
}

\end{document}